\date{today}

\documentclass[10pt, letterpaper,preprint]{amsart}

\usepackage{latexsym, amsmath, amstext, amssymb, amsfonts, amscd, bm, array, multirow, amsbsy, mathrsfs}
\usepackage{amsthm}
\usepackage{t1enc}
\usepackage[mathscr]{eucal}
\usepackage{indentfirst}
\usepackage{pb-diagram}
\usepackage{graphicx}
\usepackage{fancyhdr}
\usepackage{fancybox}
\usepackage{enumerate}
\usepackage{color}
\usepackage[all]{xy}
\usepackage{tikz}
\usepackage{xparse}
\usetikzlibrary{matrix}

\usepackage{url}
\usepackage[sort&compress,comma]{natbib}
\bibpunct{[}{]}{,}{n}{}{,}
\theoremstyle{plain}
\newtheorem{thm}{Theorem}[section]

\newtheorem{prop}[thm]{Proposition}

\newtheorem{lemma}[thm]{Lemma}
\newtheorem{cor}[thm]{Corollary}

\theoremstyle{definition}
\newtheorem{definition}[thm]{Definition}

\theoremstyle{remark}
\newtheorem{remark}[thm]{Remark}

\newcommand{\norm}[1]{\left\lVert#1\right\rVert}
\renewcommand{\a} \alpha
\renewcommand{\b} \beta

\newcommand{\Map}{\mathrm{Map}}

\newcommand{\su}{\mathfrak{su}}

\newcommand{\G}{\mathrm{G}}

\renewcommand{\S}{\Sigma}

\newcommand{\frg}{\mathfrak{g}}

\newcommand{\cA}{\mathcal{A}}
\newcommand{\cB}{\mathcal{B}}
\newcommand{\cC}{\mathcal{C}}

\newcommand{\cF}{\mathcal{F}}
\newcommand{\cG}{\mathcal{G}}

\newcommand{\cM}{\mathcal{M}}
\newcommand{\cI}{\mathcal{I}}

\newcommand{\cN}{\mathcal{N}}
\newcommand{\cL}{\mathcal{L}}
\newcommand{\cO}{\mathcal{O}}
\newcommand{\cP}{\mathcal{P}}

\newcommand{\cS}{\mathcal{S}}

\newcommand{\cU}{\mathcal{U}}

\newcommand{\cW}{\mathcal{W}}

\newcommand{\RR}{\mathbb{R}}
\newcommand{\CC}{\mathbb{C}}
\newcommand{\HH}{\mathbb{H}}

\newcommand{\PP}{\mathbb{P}}

\newcommand{\ZZ}{\mathbb{Z}}

\newcommand{\too}{\longrightarrow}
\newcommand{\x}{\times}
\newcommand{\ox}{\otimes}

\newcommand{\la}{\langle}
\newcommand{\ra}{\rangle}
\newcommand{\frM}{\mathfrak{M}}

\newcommand{\frh}{\mathfrak{h}}
\newcommand{\frs}{\mathfrak{s}}
\newcommand{\fru}{\mathfrak{u}}

\newcommand{\Id}{\mathrm{Id}}
\newcommand{\id}{\mathrm{Id}}
\renewcommand{\Re}{\mathrm{Re}}

\newcommand{\even}{\textit{even}}

\newcommand{\prim}{\textit{prim}}
\DeclareMathOperator{\rk}{rk}
\DeclareMathOperator{\Ad}{Ad}
 
\DeclareMathOperator{\im}{im}
\DeclareMathOperator{\img}{im}
\DeclareMathOperator{\coker}{coker}
\DeclareMathOperator{\Hom}{Hom}
\DeclareMathOperator{\Sym}{Sym}

\DeclareMathOperator{\End}{End} 

\DeclareMathOperator{\GL}{GL}
\DeclareMathOperator{\SO}{SO}
\DeclareMathOperator{\SU}{SU}
\DeclareMathOperator{\Spin}{Spin}
\DeclareMathOperator{\U}{U}
\DeclareMathOperator{\SL}{SL}
\DeclareMathOperator{\Cl}{Cl}

\DeclareMathOperator{\hol}{hol}
\DeclareMathOperator{\tr}{tr}

\renewcommand{\index}{\mathrm{index}\,}

\begin{document}

\title{Transversality for the moduli space of $\Spin(7)$-instantons}

\author[V. Mu\~{n}oz]{Vicente Mu\~{n}oz}
\address{Facultad de Ciencias Matem\'aticas, Universidad
Complutense de Madrid, Plaza de Ciencias 3, 28040 Madrid, Spain}
\email{vicente.munoz@mat.ucm.es}

\author[C. S. Shahbazi]{C. S. Shahbazi}
\address{Institut f\"ur Theoretische Physik, Leibniz Universit\"at Hannover, Germany}
\email{carlos.shahbazi@itp.uni-hannover.de}

\thanks{Preprint number: ITP-UH-23/16} 

\begin{abstract}
We construct the moduli space of $\Spin(7)$-instantons on a hermitian complex vector bundle over a closed $8$-dimensional manifold endowed with a (possibly non-integrable) $\Spin(7)$-structure. We find suitable perturbations that achieve regularity of the moduli space, so that it is smooth and of the expected dimension over the irreducible locus. 
\end{abstract}

\maketitle


\section{Introduction}

Gauge theory in dimensions two, three and four is by now a classical research area in geometry and topology, which has been extensively studied and developed in the literature since the seventies, with formidable results in differential topology and related areas, see for example the book \cite{DKbook} and references therein. 

Higher-dimensional gauge theory on the other hand is a much recent proposal appearing in the influential work of Donaldson 
and  Thomas \cite{DT1}, and suggests studying a higher-dimensional version of the four-dimensional instanton equations which 
exists in the presence of the appropriate geometric structure. Higher-dimensional instantons had in fact already appeared 
\emph{in disguise} in the physics literature as early as in the eighties \cite{CDFN,Ward}, although they were not systematically 
studied in the mathematical literature at the time. Recently, higher-dimensional gauge theory has experienced an increase in 
activity fueled both from pure mathematics, motivated by the program proposed by \cite{DT1,TianII,DonaldsonSegal}, as 
well as from string theory and in particular the Strominger system, see for example \cite{Harland,Ivanov:2009rh,Haupt,delaOssa:2016ivz,Clarke:2016qtg}. 
Early works in the topic include \cite{Carrion,Lewis,Figueroa}. 

Most of the literature in higher-dimensional gauge theory, and in particular the one considering instantons in eight dimensions, 
focuses on manifolds of special holonomy. For example, Lewis \cite{Lewis} constructs
$\Spin(7)$-instantons on a resolution of $T^{8}/ F$, where $F$ is a finite subgroup acting on the $8$-torus
of $T^{8}$, which has a $\Spin(7)$-holonomy structure by the results of Joyce \cite{JoyceI}. On the other hand, 
Tanaka \cite{Tanaka} constructs examples of $\Spin(7)$-instantons on the resolution of an appropriate Calabi-Yau 
four-orbifold quotiented by $\mathbb{Z}_{2}$, which is a $\Spin(7)$-holonomy manifold again by results of Joyce \cite{JoyceII}. Walpuski \cite{Walpuski}
proves an existence theorem for a particular type of $\Spin(7)$-instantons on $\Spin(7)$-holonomy manifolds 
admitting appropriate local $K3$ Cayley fibrations. 

Here instead, for reasons that will become apparent in a moment, we focus on $8$-dimensional 
manifolds equipped with a generically non-integrable $\Spin(7)$-structure. More concretely, in this article we initiate the 
construction of the moduli space of $\Spin(7)$-instantons on a hermitian complex vector bundle $E$ over a closed $8$-manifold 
$M$ equipped with a not necessarily integrable $\Spin(7)$-structure, focusing on studying the transversality properties of the 
moduli space. The existence of a $\Spin(7)$-structure on an $8$-manifold is equivalent to the existence of a $4$-form $\Omega$ 
point-wise satisfying a particular algebraic condition, but not involving any differential equation. In turn, $\Omega$ determines a 
Riemannian metric in a highly non-linear way, which is of $\Spin(7)$-holonomy if and only if $\Omega$ is closed, in which case it 
defines a calibration \cite{HarveyL}. In terms of $\Omega$, the $\Spin(7)$-instanton equation for a connection $A$ in $E$ is given by
  $$
  \ast F_{A} = -F_{A}\wedge\Omega\, ,
  $$
where $F_{A}$ is the curvature associated to $A$. We will refer to solutions of this equation as $\Spin(7)$-instantons. 
We are interested in studying the space of instantons modulo gauge transformations (automorphisms of the bundle). 
The $\Spin(7)$ instanton condition 
modulo gauge transformations is elliptic regardless the integrability properties of the underlying $\Spin(7)$-structure. 
This is in clear contrast with the situation one encounters in $7$-dimensions, and makes very natural working with 
arbitrary $\Spin(7)$-structures. In fact, the situation in $8$-dimensions regarding $\Spin(7)$-instantons is similar 
in various aspects to that in $4$-dimensions: for instance the deformation complex contains three terms
\cite{Carrion}, as it happens in $4$-dimensions, being $8$-dimensions the only case where this coincidence happens. 

Aside from what we have already said, the motivations to consider $\Spin(7)$-instantons over closed $8$-manifolds 
equipped with a not necessarily integrable $\Spin(7)$-structure are many. A practical reason comes from the existence 
of examples: explicit instances of closed $8$-manifolds of $\Spin(7)$-holonomy are scarce \cite{Joyce2007}. Relaxing 
the integrability condition we are able to enlarge the available examples. In fact, we show that every Quaternionic-K\"ahler 
closed $8$-manifold carries a necessarily non-integrable (and possibly unrelated) $\Spin(7)$-structure. This provides 
explicit examples of $\Spin(7)$-manifolds, such as $G_{2}/SO(4)$ and $\mathbb{H}{P}^{2}$. 

As explained in \cite{Haydys}, $\Spin(7)$-instantons can become a useful tool to learn about the topology of $4$-manifolds,
if we are willing to take for granted a construction assigning to every $4$-manifold $X$ a $\Spin(7)$-manifold $M_{X}$. 
For example, $M_{X}$ can be taken to be the total spinor bundle over $X$, which admits a $\Spin(7)$-structure which is 
generically non-integrable (it is integrable for example when $X=S^{4}$ \cite{BryantSalamon}). By counting then 
$\Spin(7)$-instantons on $M_{X}$, one could in principle obtain an invariant for $X$.

Let us consider now the case that $M$ is equipped with a Calabi-Yau structure $(\omega,\theta)$ and an $\SU(r)$ vector bundle $E$. 
Assuming $c_{2}(E)\in H^{2,2}(M)$, a $\Spin(7)$-instanton on $E$ is equivalent to a polystable structure on $E$, a fact 
that was already noticed in \cite{Lewis} and that is used by Tian \cite{Tian} to propose a way to attack the Hodge 
Conjecture by proving existence of $\Spin(7)$-instantons. This idea was explicitly considered and developed in reference 
\cite{Ramadas}, where Ramadas attempted to construct, without apparent success, $\Spin(7)$-instantons on some 
abelian varieties of Weil type for which Hodge Conjecture is yet to be settled. On the other hand, the first author
\cite{Munoz-JMPA,MunozII} studied, motivated by the previous proposal, under which conditions 
the existence of a polystable holomorphic structure on $E$ for $(M,\omega,\theta)$ implies the existence of a polystable 
holomorphic structure on $E$ for $(M,\hat{\omega},\hat{\theta})$, where $(\hat{\omega},\hat{\theta})$ is an appropriately
defined $\Spin(7)$-\emph{rotation} of $(\omega,\theta)$.

Last but not least, there is a strong motivation coming from string theory to consider $\Spin(7)$-manifolds equipped with a 
non-integrable $\Spin(7)$-structure. The Strominger system \cite{Strominger} is a system of PDE's on a Riemannian manifold 
that encodes the conditions for it to be an admissible supersymmetric compactification background of Heterotic supergravity. 
When formulated in $8$-dimensions, it involves a conformally balanced $\Spin(7)$-structure coupled to a $\Spin(7)$-instanton 
and a function \cite{Fernandez:2008wla,Friedrich}, and thus requires considering $\Spin(7)$-instantons with respect to 
generically non-integrable $\Spin(7)$-structures.

It should be noted\footnote{We thank Professor Dominic Joyce for explaining to us this important point.} that using generically non-integrable $\Spin(7)$-structures has some drawbacks regarding the development of Donaldson's theory in eight dimensions. More concretely, in a $\Spin(7)$-holonomy manifold there is a topological bound in the $L^2$-norm of the curvature of any $\Spin(7)$-instanton, which is an extra reason to hope that the corresponding $\Spin(7)$-instanton moduli space might admit a good compactification. However, for generic $\Spin(7)$-manifolds this is not longer true, and it is certainly possible that the $L^{2}$-norm of the curvature goes to infinity as we move in the moduli space of $\Spin(7)$-instantons, leading to a \emph{stronger} non-compactness which may indicate that such moduli space cannot be compactified and thus used to count invariants. This could be remedied by considering instead, as proposed in references \cite{DonaldsonSegal,JoyceIII}, a closed taming $\Spin(7)$-form. In this situation one again encounters a topological bound for the $L^{2}$-norm of the curvature and then one can expect a moduli space admitting a nice compactification.

\subsection*{Main results}

The main goal of this work is to study transversality for the moduli space of $\Spin(7)$-instantons. We develop the theory 
\emph{from scratch}. Section \ref{sec:Spin(7)rep} contains the necessary background on $\Spin(7)$-representations as 
well as some results on the space of $\Spin(7)$-structures on an $8$-dimensional vector space. Section \ref{sec:Spin(7)-mfds} 
contains the necessary background on $\Spin(7)$ manifolds. We prove that every closed $8$-manifold admitting an 
almost-Quaternionic structure structure admits a (possibly unrelated) $\Spin(7)$-structure. Furthermore, we give a formula for the 
Dirac operator associated to the Ivanov connection which seems to be new in the literature. Section \ref{sec:connections} 
contains a detailed analysis of some of the topological and analytic properties of the space of connections modulo gauge 
transformations. Little changes here from the situation in $4$-dimensions and in fact we follow classical references on this topic. 
However, we have chosen to include explicitly all the pertinent results and proofs, mainly because some of them are not 
explicitly proven in the literature but also in order to give a systematic and complete exposition which can serve as the foundations 
for a theory of deformations of $\Spin(7)$-instantons. This section culminates with theorem \ref{thm:localorbispace}, 
which among other things proves that the space of connections modulo gauge transformations is a topological Hausdorff space and 
gives a local description. Section \ref{sec:moduli-local} studies the local structure of the moduli space of $\Spin(7)$-instantons. 
The main result of this section is theorem \ref{thm:main1} which gives the local model for the moduli space in terms of 
the hypercohomology groups of the appropriate deformation complex. Section \ref{sec:transversality} addresses the 
transversality properties of the moduli space of $\Spin(7)$-instantons on a complex hermitian vector bundle of any rank, considering 
first the rank-two case. In particular, theorem \ref{thm:projectorpert} proves that in the rank-two case, for a dense 
family of projector perturbations the corresponding moduli spaces are smooth at irreducible connections and of the expected dimension. 
In the higher-rank case, theorem \ref{thm:holonomypert} proves that for a dense family of holonomy perturbations, 
the perturbed moduli spaces are again smooth at irreducible connections and of the expected dimension. 
Section \ref{sec:instantonline} considers explicitly the moduli space of $\Spin(7)$-instantons on a principal $\mathrm{U}(1)$-bundle. 
The main result of this section is theorem \ref{thm:U(1)moduli}, which characterizes the moduli space of $\mathrm{U}(1)$-instantons for a 
generic $\Spin(7)$-structure under a relatively mild assumption on its torsion.

\subsection*{Acknowledgements} We would like to thank S. Donaldson, N. Hitchin, D. Joyce, O. Lechtenfeld, J.J. Madrigal,  R. Thomas, for useful conversations. 
First author was partially supported through Project MICINN (Spain) MTM2015-63612-P. The second author was partially supported by the German Science Foundation (DFG) Project LE838/13.

\section{Representation theory of the group $\Spin(7)$} \label{sec:Spin(7)rep}

 On $\RR^{8}$, with coordinates $(x_1,\ldots, x_8)$, we consider the $4$-form:
\begin{align} \label{eq:Omega0}
\Omega_{0}  =& \, dx_{1234} - dx_{1278} - dx_{1638} - dx_{1674} + dx_{1526} + dx_{1537} + dx_{1548} \nonumber\\ 
&+ dx_{5678} - dx_{5634} - dx_{5274} - dx_{5238} + dx_{3748} + dx_{2648} + dx_{2637}\, ,
\end{align}
where $dx_{abcd}$, $a, b, c, d =1,\hdots, 8$, stands for $dx_{a}\wedge dx_{b}\wedge dx_{c}\wedge dx_{d}$. The
subgroup of $\GL(8,\RR)$ that fixes $\Omega_{0}$ is isomorphic to $\Spin(7)$, which is a simply-connected, compact, 
Lie group of dimension $21$, abstractly isomorphic to the double cover of $\SO(7)$. This group also preserves the 
standard orientation of $\RR^{8}$ and the euclidean metric, hence  $\Spin(7)\subset \SO(8)$. Also, it is easy to see
that $\Omega_{0} = \ast \Omega_{0}$, where $\ast$ is the Hodge dual. 

Consider an oriented $8$-dimensional vector space $V$. A $\Spin(7)$-form is a $4$-form $\Omega$ that can be written as $\Omega_0$ in suitable coordinates, i.e., there exists an orientation-preserving isomorphism $f : V \to\RR^{8}$ such that $\Omega = f^{\ast} \Omega_{0}$. The space $\cS$ of $\Spin(7)$-forms is thus a $43$-dimensional homogeneous subspace of $\Lambda^4$:
$$
\cS \cong \GL^+(8,\RR) /\Spin(7)\, .
$$
The space $\cS_{\nu}$ of $\Spin(7)$-forms compatible with a given volume form $\nu$ is diffeomorphic to the homogeneous space $\SL(8,\RR)/\Spin(7)$, whereas the space $\cS_g$ of $\Spin(7)$-forms compatible with a given Riemannian structure $g$ is diffeomorphic to the homogeneous space $\SO(8)/\Spin(7)$, which is of dimension $7$. If on the other hand we only fix the conformal structure $c=[g]$, the corresponding space of compatible $\Spin(7)$-forms is $\cS_c \cong (\RR_+\cdot \SO(8))/\Spin(7)$.

There is a different characterization of $\Spin(7)$ as the stabilizer of an element in one of the three irreducible representations of $\Spin(8)$ on an eight-dimensional vector space $V$. For this, fix an orientation and a metric for $V$. Consider the Clifford algebra $\Cl(8)$ associated to it, and recall that the group $\Spin(8)$ satisfies $\Spin(8)\subset \Cl^{\even}(8)\subset \Cl(8)$. 
The $16$-dimensional irreducible representation $S$ of $\Cl(8)$ admits a unique 
 bilinear $\la $-$ ,$-$ \ra \colon S\times S\to \mathbb{R}$ for which Clifford multiplication is orthogonal.
It splits into two $8$-dimensional irreducible inequivalent representations $S=S^+ \oplus S^-$ of $\Cl^{\even}(8)$. This produces two inequivalent representations $\gamma^\pm: \Spin(8) \to \SO(S^\pm)$ of $\Spin(8)$. In addition, there is a third eight-dimensional representation of $\Spin(8)$ given by its adjoint action\footnote{Meanning the adjoint action of $\Spin(8)$ on $V\subset\Cl(8)$ via the Clifford algebra product. Not to be confused with the usual adjoint representation of a Lie group.} on $V\hookrightarrow \Cl(8)$, which we denote by $\mathfrak{D}: \Spin(8) \to \SO(V)$. The triality automorphism is an outer automorphism of $\Spin(8)$ that permutes these three representations. All of the three representations are representations of the (universal) double cover of $\SO(8)$. Clifford multiplication constitutes a $\Spin(8)$-equivariant map $c\colon V \ox S^+ \to S^-$. 

The group $\Spin(7)$ can be now defined as the stabilizer of a unit-norm element $\eta$ in either of $S^+$, $S^-$ or $V$. This gives three conjugacy classes of subgroups $\Spin(7)$ inside $\Spin(8)$, which are cyclically permuted by the triality automorphism of $\Spin(8)$. The one that agrees with the previous definition is given by fixing an element $\eta\in S^7\subset S^+$.

In the above notation, the standard representation $V=\RR^8$ of $\SO(8)$ induces an $8$-dimensional representation
under the inclusion $\Spin(7)<\Spin(8)$ followed by the adjoint
representation $\Spin(8) \to \SO(8)$. 
The positive spin representation of $\Spin(8)$ splits in $\Spin(7)$-representations as:  
$$ 
S^+=\la\eta\ra \oplus H \, ,
$$
where $\la\eta\ra$ denotes the one-dimensional trivial representation of $\Spin(7)$ and $H$ denotes the seven-dimensional representation isomorphic to the standard representation of $\SO(7)$, induced by the adjoint representation of the double cover $\Spin(7)\to \SO(7)$. The $\Spin(7)$-equivariant map:
\begin{eqnarray*}
c\colon V \ox S^+ &\to& S^-\, ,\\
v\otimes\eta &\mapsto& v\cdot \eta\, ,
\end{eqnarray*}
gives an isomorphism $V \to S^-$, via Clifford multiplication by $\eta$. 

The $\Spin(7)$-representations on the exterior powers of $V$ are as follows. Denote by $\Lambda^{i} = \Lambda^{i} V$, $ i =1,\hdots, 8$. We have the following decompositions of the representation $\Lambda^{i}$ into irreducible $\Spin(7)$ factors \cite{Munoz-JMPA}:
  \begin{align*} 
  \Lambda^{1} &= \Lambda^{1}_{8}\, , \\ 
  \Lambda^{2} &=\Lambda^{2}_{7} \oplus \Lambda^{2}_{21}\, , \\
  \Lambda^{3} &=\Lambda^{3}_{8} \oplus \Lambda^{3}_{48}\, , \\
  \Lambda^{4} &= \Lambda^{4}_{+} \oplus \Lambda^{4}_{-}\, , \\
  \Lambda^{4}_{+}  &= \Lambda^{4}_{1} \oplus \Lambda^{4}_{7} \oplus \Lambda^{4}_{27},  \\
  \Lambda^{4}_{-} &= \Lambda^{4}_{35}.
 \end{align*}
Here $\Lambda^{i}_{j}$ denotes the irreducible subrepresentation of $\Lambda^{i}$ of dimension $j$. As mentioned above, $\Lambda^1_8=V$, and $\Lambda^2_7 \cong H$. The second isomorphism is given by Clifford multiplication: 
 \begin{align}  \label{eqn:cI} 
 \cI\colon\Lambda^2_7  & \too H\, ,  \nonumber \\
 \alpha & \mapsto \alpha \cdot \eta\, ,
 \end{align}
The fact that $\cI$ is an isomorphism follows from $\Spin(7)$-equivariance together with the identity $\la \alpha\cdot\eta,\eta\ra
= 0$, for all $\alpha\in \Lambda^{2}$. 

The subspace $\Lambda^2_{21}$ is the space associated to the Lie algebra 
$\mathfrak{spin}(7)\simeq \mathfrak{so}(7)\subset \mathfrak{so}(8) = \Lambda^2$. Moreover, we have the eigenvalue decomposition:
  \begin{align*}
  \Lambda^2_7 &= \{ \alpha \in \Lambda^2 \, | \, *(\alpha\wedge \Omega)= 3\alpha\}, \\
  \Lambda^2_{21} &=\{ \alpha \in \Lambda^2 \, | \, *(\alpha\wedge \Omega)= - \alpha\}.
  \end{align*}
For three-forms, we have: 
  \begin{align*}
 \Lambda^3_8 &=\{ *(\alpha \wedge \Omega) | \alpha\in \Lambda^1_8\}, \\ 
 \Lambda^3_{48} &=\{\beta \in \Lambda^3 | \beta\wedge\Omega=0\}.
  \end{align*}
Regarding four-forms, we have that $\Lambda^4_{\pm}$ are the eigenspaces of the Hodge star operator $*$ on $\Lambda^4$, both of dimension $35$. It can be seen that $\Lambda^4_1=\la \Omega\ra$. For describing $\Lambda^4_7$, consider 
$\Lambda^2_7 \subset \Lambda^2\subset V\ox V  \cong V \ox V^*$, and take the image of $\Lambda^2_7  \ox \Lambda^4_1 
\to  V \ox V^* \ox \Lambda^4 V \to \Lambda^4 V$, by contracting $(v\ox \Theta) \ox \alpha \mapsto v\wedge i_\Theta \alpha$.

\begin{prop}\label{prop:wedges}
 Under the wedge map, we have the following:
\begin{align*}
  \Lambda^2_7 \x \Lambda^2_7 & \too  \Lambda^4_{27} \oplus \Lambda^4_1  , \\
  \Lambda^2_7 \x \Lambda^2_{21} & \too  \Lambda^4_7 \oplus \Lambda^4_{35}  , \\
  \Lambda^2_{21} \x \Lambda^2_{21} & \too \Lambda^4_1 \oplus \Lambda^4_{27}  \oplus \Lambda^4_{35}.
\end{align*}
\end{prop}

\begin{proof}
The first and second item appear in Subsection 2.1 and Remark 3 of \cite{Munoz-JMPA}.

The last one is proved in an analogous fashion to \cite{Munoz-JMPA}. We consider the $\Spin(7)$-structure induced
by an $\SU(4)$-structure under the inclusion $\SU(4)\subset \Spin(7)$. The $\SU(4)$-structure is given by a complex
structure $J$ on $V$, a K\"ahler form $\omega$ and a $(4,0)$-form $\theta$, by setting
$\Omega=\frac12 \omega^2+ \Re \, \theta$. The complex structure allows to define the spaces of $(p,q)$-forms
$\Lambda^{p,q}\subset \Lambda^{p+q}_\CC$. We denote $\triangle^{p,q}=\Re(\Lambda^{p,q})$, and use
the sub-index $\prim$ to denote the subspace of primitive forms. By \cite{Munoz-JMPA}, 
$\Lambda^2_{21}=A_-\oplus \triangle^{1,1}_{prim}$, where $A_\pm$ are defined as the eigensapces of
$\triangle^{2,0}$ of the anti-linear complex Hodge operator $*_\theta$ defined by $\theta$. To characterize the image of the wedge map, 
it is enough to see where it lies the image of $v\wedge v$ for an element $v$, since the collection of them span the image
of $\Lambda^2_{21}\x \Lambda^2_{21}$. For this we can use an element $v\in \triangle^{1,1}_{prim}$, which
has image in $\triangle^{2,2}=\triangle^{2,2}_{\prim}\oplus \triangle^{1,1}_{\prim}\omega\oplus \la \omega^2\ra$.
Also, in \cite{Munoz-JMPA} it is proved that $\Lambda^4_7=A_-\omega \oplus \la \im \, \theta\ra$, therefore
all the three
summands of $\triangle^{2,2}$ lie in $\Lambda^4_1 \oplus \Lambda^4_{27}  \oplus \Lambda^4_{35}$. To see 
that the three of them appear, note that the component of $v\wedge v$ in 
$\Lambda^4_1$ is non-zero, since $\Sym^2(\Lambda^2_{21})$ has an invariant
quadratic form. Also it is easy to write a pair of forms in $\triangle^{1,1}_{\prim}$ whose wedge
is in $\triangle^{2,2}_{\prim}$ (i.e. $dz_{1\bar2}$, $dz_{3\bar4}$), and also a pair of forms whose wedge is
in $\triangle^{1,1}_{\prim}\omega$ (i.e. $dz_{1\bar1}-\frac14\omega$, and itself). So all three components appear.
\end{proof}

\begin{remark}\label{rem:ccc}
If $\alpha\in \Lambda^2_7$, $\beta\in \Lambda^2_{21}$, and $\alpha\wedge\beta=0$, then either $\alpha=0$ or $\beta=0$.
To prove this, suppose that $\alpha\neq 0$. Then we can consider an $\mathrm{SU}(4)\subset \Spin(7)$-structure such that
$\omega=\alpha$. With respect to this $\mathrm{SU}(4)$-structure, we have $\Lambda^2_{21}=A_-\oplus \triangle_{prim}^{1,1}$.
But $\omega A_- \oplus \omega \triangle^{1,1}_{prim}\subset \Lambda^4$ is a $21$-dimensional vector subspace.
So $\omega\wedge \beta=0\implies \beta=0$.
\end{remark}

Now we want to have a closer look at the space $\cS \subset \Lambda^4$ of all $\Spin(7)$-forms, and
the spaces $\cS_g, \cS_{\nu}, \cS_c$ defined above. 

\begin{prop}\label{prop:tangent}
Consider $\Omega\in \cS$. We have the following tangent spaces at $\Omega$:
\begin{align*}
 T_\Omega \cS &= \Lambda^4_1 \oplus \Lambda^4_7 \oplus \Lambda^4_{35}, \\
 T_\Omega \cS_g &= \Lambda^4_7 , \\
 T_\Omega \cS_{\nu} &= \Lambda^4_7 \oplus \Lambda^4_{35}, \\
 T_\Omega \cS_c &= \Lambda^4_1 \oplus \Lambda^4_7 . 
\end{align*}
\end{prop}

\begin{proof}
The space $\cS$ is diffeomorphic to a homogeneous space $\cS \cong \GL^+(8,\RR)/\Spin(7)$, with $\Omega$ corresponding
to the class of the identity $\Id\in \GL^+(8,\RR)$. 
The tangent space $T_\Omega\cS$ carries the isotropy representation of the stabilizer at $\Omega$ of the $\GL^+(8,\RR)$-action on $\cS$, which is isomorphic to $\Spin(7)$. Hence the tangent space $T_\Omega\cS$ becomes a $\Spin(7)$-representation, and the action of $\Spin(7)$ on the four-forms correspond with the adjoint action on the tangent space $T_{\Id} ( \GL^+(8,\RR)/\Spin(7))$. This means that $T_\Omega \cS$ is a $\Spin(7)$-subrepresentation of $\Lambda^4$. As it is of dimension $43$, the result follows. The other items are analogous.
\end{proof}

\section{$\Spin(7)$-manifolds} \label{sec:Spin(7)-mfds}

Let $M$ be an oriented $8$-dimensional manifold. For each point $p\in M$ we denote by $\S_{p} \subset \Lambda^{4} T^{\ast}_{p} M$ 
the set of all $\Spin(7)$-forms at $p$, namely $\Omega_{p}\in \S_{p}$ if there exists an oriented isomorphism 
$f_{p}\colon T_{p}M \to \RR^{8}$  such that $\Omega_{p} = f^{\ast} \Omega_{0}$, where $\Omega_{0}$ is the canonical $4$-form 
defined in equation \eqref{eq:Omega0}. We denote by $\S(M)$ the fiber bundle over $M$ with fiber given by $\S_{p}$, for each $p\in M$. 
Then, global sections of $\S(M)$ are by construction in one-to-one correspondence with reductions of the frame bundle of $M$ from 
$\GL^{+}(8,\RR)$ to $\Spin(7)$. We define $\cS(M):= \Omega^{0}(\Sigma(M))$ to be the space of smooth sections of $\Sigma$.

\begin{definition} \label{def:Spin(7)-structure}
A $\Spin(7)$-structure on an $8$-dimensional manifold $M$ is a reduction of the frame bundle $F(M)$ of $M$ to $\Spin(7)$.
That is, a $\Spin(7)$-structure is a choice of a $4$-form $\Omega\in \Omega^4(M)$ such that $\Omega_p\in \Sigma_p$, for each $p\in M$.
\end{definition}

We have that $\cS(M)$ is then the space of $\Spin(7)$-structures on $M$. The existence of a $\Spin(7)$-structure $\Omega$ on $M$ allows for a point-wise decomposition of $\Lambda^{i}T^{\ast}M$, $i=1,\hdots, 8$, in $\Spin(7)$-representations, which we denote with a subscript as described in section \ref{sec:Spin(7)rep}. We define
 $$
 \Lambda^{i}_{k}(M) := \Lambda^{i}_{k}\, T^{\ast}M\, , \qquad \Omega^{i}_{k}(M) := \Gamma(\Lambda^{i}_{k}T^{\ast} M)\,  ,
 $$
where $k$ denotes the specific $\Spin(7)$-representation. A $\Spin(7)$-structure determines an orientation and a riemannian metric on $M$. An oriented $8$-dimensional manifold $M$ admits a $\Spin(7)$-structure (compatible with that orientation) if and only if $M$ is spin and in addition (cf.\  Theorem 10.7 in \cite{Spingeometry})
\begin{equation} \label{eq:topcond1}
p_{1}(M)^2 - 4 p_{2}(M) + e (M) = 0\, ,
\end{equation} 
where $p_{1}(M)$, $p_{2}(M)$ are the Pontrjagin classes of $M$ and $e(M)$ is the Euler class of $M$.

Equivalently, we can characterize $\Spin(7)$-structures by using non-zero spinors on $M$. We fix an orientation and a Riemannian metric, so that we have a frame bundle with structure group $\SO(8)$. Recall that an eight-manifold is spin if and only if the frame bundle can be lifted to a $\Spin(8)$-bundle $P_{\Spin(8)}(M)$ in a compatible way with the double covering $\mathfrak{D}\colon \Spin(8)\to \SO(8)$. The obstruction for an orientable manifold to be spin is given by its second Stiefel-Whitney class. Assuming that $M$ is spin, we  can equip $M$ with two spinor bundles $S^\pm$, which are associated to $P_{\Spin(8)}(M)$ through the two eight-dimensional irreducible inequivalent representations $\gamma^\pm: \Spin(8) \to \SO(S^\pm)$ of $\Spin(8)$ and of chirality $\pm$. 

\begin{prop}\cite[Theorem 10.7]{Spingeometry} \label{prop:spinequivalence}
An oriented $8$-dimensional manifold $M$ is $\Spin(7)$ if and only if it is spin and carries a unit-norm spinor $\eta\in \Gamma(S^{+})$.
\end{prop}

We will call such unit spinor $\eta\in \Gamma(S^{+})$ the \emph{associated spinor} to the $\Spin(7)$-structure $\Omega$. as the The left hand side in equation \eqref{eq:topcond1} is equal to $e(S^+)$. Since the rank of $S^{+}$ equals the dimension of $M$, the existence of a nowhere zero spinor $\eta\in \Gamma(S^+)$ is equivalent to the vanishing $e(S^+)=0$. 

Let $(M,\Omega)$ be a $\Spin(7)$-manifold, and let $g$ be the induced Riemannian metric and $\nabla$ the Levi-Civita connection. The $\Spin(7)$-structure is integrable, that is, the holonomy of $(M,g)$ is contained in $\Spin(7)<\SO(8)$, if and only if $\nabla \Omega=0$. 
By \cite{Fernandez-Gray,FernandezSpin(7)}, this is equivalent to $d\Omega=0$. Note that $\ast\Omega=\Omega$, so that in this case $\Omega$ is closed and co-closed. Examples of compact manifolds with $\Spin(7)$-holonomy are relatively scarce. The first examples were given by D.\ Joyce \cite{Joyce2007}. If we relax the requirement of having $\Spin(7)$-holonomy and we allow general $\Spin(7)$-structures, then there are more examples. 
We find useful to point out the following result, which seems to have passed unnoticed in the literature.

\begin{prop}
\label{prop:AQK}
Let $M$ be a closed $8$-manifold admitting an almost-Quaternionic structure on its tangent space. Then, $M$ admits a $Spin(7)$-structure $\Omega\in \cS(M)$, in general unrelated to the existent almost-Quaternionic structure.
\end{prop}

\begin{proof}
It follows from corollary 8.3 in \cite{CadekVanzura} together with theorem 10.7 in \cite{Spingeometry}.
\end{proof}

Proposition \ref{prop:AQK} automatically provides a relatively large number of explicit manifolds carrying $\Spin(7)$-structures, which, to the best of our knowledge have not been studied or characterized. In particular, proposition \ref{prop:AQK} shows that the eight-dimensional quaternionic space $\HH\PP^2$ carries a $\Spin(7)$-structure. For an explicit early example of a compact eight-manifold carrying a non-integrable $\Spin(7)$-structure the reader may consult \cite{FernandezExample}.

For a $\Spin(7)$-structure $\Omega$, we define its torsion as \cite{FernandezSpin(7)}:
 $$
 W := d\Omega  \in \Omega^5(M).
 $$
As $\Omega^5(M) = \Omega^5_8(M) \oplus \Omega^5_{48}(M)$ in irreducible $\Spin(7)$-representation, we have the orthogonal decomposition 
$W = W_{8} \oplus W_{48}$, where  $W_8\in \Omega^5_8(M)$ and $W_{48}\in \Omega^5_{48}(M)$. 
We define the Lee form of the $Spin(7)$-structure as
 $$
 \theta = \ast (d^{\ast}\Omega \wedge\Omega)\, . 
 $$
Then reference \cite{FernandezSpin(7)} distinguishes four types of $\Spin(7)$-structures:

\begin{itemize}
\item $\Spin(7)$-holonomy structures, defined by satisfying $W_{8} = W_{48} = 0$.
\item Balanced $\Spin(7)$-structures, defined by having vanishing Lee-form, $\theta = 0$. So $W_8=0$.
\item Locally conformally parallel $\Spin(7)$-structures, defined by the condition $d\Omega =-\frac{1}{7} \theta\wedge\Omega$. So $W_{48}=0$.
\item Generic $\Spin(7)$-structures, with no specific restriction on $W_{8}$ or $W_{48}$.	
\end{itemize}

If we define a $\Spin(7)$-structure via a unit spinor $\eta$ as in proposition \ref{prop:spinequivalence}, then the $\Spin(7)$-structure is integrable if and only if $\nabla \eta=0$, where $\nabla$ is the spin Levi-Civita connection. In the non-integrable case, there exists a canonical metric-compatible connection with torsion $\nabla^{T}\colon \Omega^{0}(S^{+})\to \Omega^{1}(S^{+})$ which preserves $\eta$, i.e., satisfies $\nabla^{T}\eta = 0$.

\begin{thm} \label{thm:Spin(7)H}
 Let $(M,\Omega)$ be a $\Spin(7)$-manifold with $\Spin(7)$-structure $\Omega$ and associated spinor $\eta\in\Omega^{0}(S^{+})$. 
  The Cayley four-form    $\Omega$ and the spinor $\eta$ are related as follows:
\begin{equation} \label{eq:etaOmeganu}
\eta\otimes \eta = 1 + \Omega + \nu\, , 
\end{equation}
where $\ast 1 = \nu\in\Omega^{8}(M)$. In particular
\begin{equation}
\label{eq:etaOmega}
\Omega(u,v,w,z) = \frac{1}{4!} \la (u\wedge v\wedge w\wedge z)\cdot\eta, \eta\ra \, , \qquad u, v, w, z \in \mathfrak{X}(M)\, .
\end{equation}
Furthermore, there exists a unique connection $\nabla^{T}$ with fully antisymmetric torsion $T$ such that $\nabla^{T}\eta = 0$. The torsion is given by
\begin{equation}
\label{eq:etatorsion}
T = - d^*\Omega - \ast\left( \theta \wedge \Omega\right)\, , \qquad \theta = \frac{1}{6}  \ast\left( d^*\Omega\wedge \Omega\right)\, ,
\end{equation}
and it acts on $\eta$ through Clifford multiplication as $T\cdot \eta = -\theta \cdot \eta$. 
\end{thm}

\begin{proof}
Equations \eqref{eq:etaOmeganu} and \eqref{eq:etaOmega} follow from \cite[Theorem 10.18]{Spingeometry}. Equation \eqref{eq:etaOmeganu} 
should be interpreted as follows: identifying $S\cong S^*$ by means of the bilinear product, we have an element $\eta\otimes \eta \in 
S\otimes S \cong S \otimes S^* =\End(S) \cong \Cl(M,g)$. But $\Cl(M,g) \cong \Lambda^* T^*M$ as vector spaces. Then $\eta\otimes \eta$ is mapped
to the poly-form $1+\Omega+\nu$.

The fact that there exists a unique connection $\nabla^{T}$ satisfying $\nabla^{T}\eta = 0$ together with equation \eqref{eq:etatorsion} and the Clifford action of $T$ on $\eta$ follow from \cite[Theorem 1.1]{IvanovSpin(7)}. 
\end{proof}

\begin{remark}
The relation between $T$ and $W$ can be extracted from equation \eqref{eq:etatorsion} and it is relatively involved:
 $$
 \ast T =  W + \frac{1}{6} \left(\ast\left( \ast W \wedge \Omega\right) \wedge \Omega\right)\, .
 $$
In particlar, $T_8=0 \Leftrightarrow W_8=0$ and $T_{48}=0 \Leftrightarrow W_{48}=0$.
\end{remark}

We will call $\nabla^T$ the \emph{Ivanov connection} associated to the $\Spin(7)$-structure $\Omega$. For later use, we want to consider in more detail the properties of the Dirac operator associated to the Ivanov connection on the spinor bundle of a $\Spin(7)$-manifold $(M,\Omega)$ as well as the corresponding index theorem. As usual, we will denote by $\eta\in\Gamma(S^{+})$ the spinor corresponding the $\Spin(7)$-structure $\Omega$. 

\begin{remark}
Let us describe the integrability condition in terms of $\nabla \eta$ and $\nabla \Omega$ and compare both tensors. 
First $\nabla \eta \in \Lambda^1 \otimes S^+$. But as it is $\la \nabla \eta ,\eta\ra=0$, we have $\nabla \eta \in \Lambda^1\otimes H$.  This
produces an element $\cI^{-1}(\nabla \eta)\in \Lambda^1\otimes \Lambda^2_7$, via (\ref{eqn:cI}). To find it explictly, note that 
the Ivanov connection is given by $\nabla_X^T Y= \nabla_X Y + \frac12 T(X,Y)$, in terms of the Levi-Civita connection. From this, it
follows that $\nabla_X \eta= -\frac12 T(X,\,$-$)\cdot \eta$. So $\cI^{-1}(\nabla \eta)=-\frac12 T$. Note also that the wedge map (i.e., anti-symmetrization)
gives an isomorphism $\Lambda^1\otimes \Lambda^2_7 \to \Lambda^3$.

Second $\nabla \Omega \in \Lambda^1\otimes \Lambda^4$. But $\nabla_X \Omega$ gives the variation of $\Omega_x$, for $x$ moving in the 
direction of $X$. As $T_\Omega \cS_g=\Lambda^4_7$, by proposition \ref{prop:tangent}, we have that $\nabla_X\Omega\in \Lambda^4_7$.
So $\nabla\Omega \in \Lambda^1\otimes \Lambda^4_7$. Again, the wedge map $\Lambda^1\otimes \Lambda^4_7 \to \Lambda^5$ is an
isomorphism, and $\nabla \Omega$ is mapped to $W=d\Omega$. In particular, we recover that $\nabla\Omega=0 \Leftrightarrow d\Omega=0$.
\end{remark}

As we have already explained, since $M$ is an $8$-dimensional oriented spin manifold, the bundle of irreducible Clifford modules $S$ admits the $\mathbb{Z}_{2}$-grading $S = S^{+}\oplus S^{-}$, given by the volume form $\nu$, which is parallel, squares to plus one and it is central in $\Cl^{\even}(M,g)$. 
Let $E$ be a real vector bundle over $M$. Then $S\otimes E = (S^{+}\otimes E) \oplus (S^{-}\otimes E)$ automatically becomes a $\mathbb{Z}_{2}$-graded bundle of real Clifford modules over $M$. 

Let $\nabla_{A}$ be a connection on $E$ and let $\nabla^{T}$ be the Ivanov spin connection. Associated to $\nabla = \nabla^{T}\otimes 1 + 1 \otimes \nabla_{A}$ we consider the Dirac operator
 \begin{equation}
 D^{\pm}_{T}\colon \Omega^{0}(S^{\pm}\otimes E)\to \Omega^{0}(S^{\mp}\otimes E)\, .
 \end{equation}
By the Index theorem we have:
 $$
 \mathrm{Ind}\, D^{+}_{T} = \mathrm{Ind}\, D^{-}_{T} = \left\{\mathrm{ch}\, E \cdot \hat{\mathrm{A}}(M)\right\}\left[ M\right]\, ,
 $$
where $\mathrm{ch}\, E$ denotes the Chern-character of $E$ and $\hat{\mathrm{A}}(M)$ denotes the \emph{A-roof genus} of $TM$.

Let us recall that in a $\Spin(7)$-manifold with $\Spin(7)$-structure given by a positive-chirality spinor $\eta$ the following isomorphisms hold
\begin{equation}
\label{eq:isospinoforms}
S^{+} \cong  \Lambda^0(M) \oplus \Lambda^{2}_{7}(M) \, , \qquad S^{-} \simeq \Lambda^{1}_{8}(M)\, ,
\end{equation}
where $\Lambda^0(M)$ is the trivial line bundle over $M$.

\begin{prop} \label{prop:DiracLA}
Through the isomorphisms \eqref{eq:isospinoforms}, the Dirac operator $D^{-}_{T}$ acts on $\Omega^{1}(M)$ as follows:
\begin{eqnarray*}
D^{-}_{T}\colon \Omega^{1}(E) &\to & \Omega^{0}(E)\oplus \Omega^{2}_{7}(E)\, ,\nonumber\\
\tau &\mapsto &  d^{\ast}_{A}\tau \oplus \pi_{7}\left( d_{A}\tau + \iota_{\tau}T\right)
\end{eqnarray*}
where $\pi_7:\Lambda^2\to\Lambda^2_7$ is the orthogonal projection.
\end{prop}

\begin{proof}
The isomorphism between $S^{-}\otimes E$ and $\Lambda^{1}\otimes E$ is given by
 \begin{eqnarray*}
 F\colon \Omega^{1}(E) &\to & \Omega^{0}(S^{-}\otimes E)\, , \nonumber\\
 \tau &\mapsto & \tau\cdot\eta\, .
 \end{eqnarray*}
Therefore, for every $\Theta \in \Omega^{0}(S^{-}\otimes E)$ there exist a unique $\tau \in \Omega^{1}(E)$ 
such that $\Theta = \tau\cdot\eta$. Let $\left\{ e^{i}\right\}$ be a local coframe and let $\left\{ e_{i}\right\}$ be the corresponding local frame
of $TM$. We have	
 \begin{eqnarray}\label{eq:DTT}
   D^{-}_{T}(\Theta) &=& D^{-}_{T}(\tau\cdot\eta) \nonumber\\
 &=& e^{i}\cdot \nabla_{e_{i}} \tau \cdot \eta \nonumber\\
 & =& e^{i}\cdot \nabla_{e_{i}} \tau \cdot \eta \\
 &=& \left(e^{i}\wedge \nabla^{T}_{e_{i}}\tau\right)\cdot\eta -\la e^{i}, \nabla^{T}_{e_{i}}\tau \ra \eta \nonumber \\ 
&=&  d^{\ast}_{A}\tau\cdot\eta + d_{A}\tau\cdot\eta  + \iota_{\tau} T\cdot\eta\, . \nonumber
\end{eqnarray} 
Here we have used that $\nabla^{T}\eta = 0$ as well as
$d_{A}\tau = e^{i}\wedge \nabla_{e_{i}} \tau + \iota_{\tau} T$, and $d^{\ast}_{A}\tau = -\iota_{e_{i}} \nabla_{e_{i}}\tau$. The
latter one needs that $i_{e_i}i_{e_i}i_\tau T=0$, because $T$ is fully antisymmetric.
From equation \eqref{eq:DTT} we finally obtain	
 $$
 D^{-}_{T}(\tau) = d^{\ast}_{A}\tau \oplus \pi_{7}\left(d_{A}\tau + \iota_{\tau} T\right)\, ,
 $$
for every $\tau\in \Omega^{1}(E)$.
\end{proof}

\section{Analytic properties of the group of gauge transformations} \label{sec:connections}

Let  $(M,\Omega)$ be an $8$-dimensional manifold with an $\Spin(7)$-structure $\Omega$ and let $P$ be a principal $\G$-bundle over $M$, where $\G$ is a compact, semi-simple Lie group whose Lie algebra we denote by $\mathfrak{g}$. Associated to $P$ we consider a complex vector bundle $E = P\times_{\rho} \mathbb{C}^{r}$ of rank $r$, where $\rho$ denotes a $r$-dimensional faithful irreducible complex representation of $\G$. We denote by $\frg_{E}\subset \End(E)$ the bundle of endomorphisms of $E$ associated to the adjoint bundle of algebras $\mathrm{ad}(P) = P\times_{\Ad} \frg$ of $P$.

\begin{remark}
We will be mainly interested in the case $\G = \U(r)$, whose Lie algebra we denote by $\mathfrak{u}(r)$. In this case we will denote by $\fru_{E}\subset \End(E)$ the bundle of skew-hermitian endomorphisms of $E$, whereas when necessary we will denote by $\mathfrak{su}_{E}\subset \End(E)$ the bundle of trace-less skew-hermitian endomorphisms of $E$.
\end{remark}

We will denote by $\cA$ the space of $\G$-compatible connections on $E$. For $A\in \cA$, we denote by $F_A\in\Omega^{2}(\frg_{E})$ its curvature. In addition, we denote by $\pi_7:\Lambda^2(M) \to \Lambda^2_7(M)$ and $\pi_{21}:\Lambda^2(M) \to \Lambda^2_{21}(M)$ the orthogonal projections onto the respective summands. 

The group of gauge transformations  $\cG$ is defined as the group of all differentiable automorphisms of $E$ or, equivalently, as the space $\Omega^{0}(\Ad(P))$ of all differentiable sections of the bundle $\Ad(P) = P\times_{\mathrm{Conj}} \G$, where $\G$ acts on itself by conjugation. A third, equivalent, description of $\cG$ is given by
\begin{equation*}
\cG = \Map_{\G}(P, \G)\, ,
\end{equation*}
i.e., by the space of differentiable maps from $P$ to $\G$ which are $\G$-equivariant with respect to the adjoint action of $\G$ on itself. 


\begin{definition}
We define the {\bf reduced gauge group} $\bar{\cG}$ as $\bar{\cG} = \cG/Z(\G)$, where $Z(\G)$ denotes the center of $\G$.
\end{definition}

	



As it has been defined, the gauge group and the reduced gauge group have only the abstract structure of a group, not even a topological group. The gauge group can be made into a topological group by endowing it with the $C^{\infty}$ compact-open topology. However, this is not the topology that we will use in this paper. In order to proceed further, we need to complete $\cG$ and $\cA$ using suitable Sobolev norms. These completions will induce the appropriate topological and metric structures on the corresponding completed spaces.   The following results can be found  in \cite{MitterViallet}.

\begin{itemize}	
	\item We denote by $\Omega^{0}_{s+1}(\End E)$ the Sobolev completion of $\Omega^{0}(\End E)$ with respect to the Sobolev norm $L^2_{s+1}$.  The $L^2_s$-norm of an element $f$ will be denoted $||f||_{s}$.
For $s > \frac{1}{2} \dim(M)$ the Sobolev continuous embedding theorem implies
that  $\Omega^{0}_{s+1}(\End E)\subset C^{0}(\End E)$ is a compact continuous embedding. Point-wise multiplication is well-defined and continuous in $\Omega^{0}_{s+1}(\End E)$. We define ${\cG}_{s+1}$ to be the Sobolev completion of ${\cG}$ respect to the Sobolev norm $L^2_{s+1}$, obtained by considering $\cG$ as a subspace of $\Omega^{0}(\End E)$. Hence $\cG_{s+1}\subset \Omega^{0}_{s+1}(\End E)$ as a closed subspace. 
We give ${\cG}_{s+1}$ the subspace topology induced by $\Omega^{0}_{s+1}(\End E)$.
For $s>\frac{1}{2}\dim\, M$ we have that ${\cG}_{s+1}$ is an infinite-dimensional smooth Hilbert-Lie group with respect to the topology given by the Sobolev norm $L^2_{s+1}$. In a similar way we Sobolev-complete $\bar{\cG}$, obtaining $\bar{\cG}_{s+1}$.	
	\item As an infinite-dimensional Hilbert-Lie group, the Lie algebra $T_{\Id}(\cG_{s+1})$ of $\cG_{s+1}$ can be identified with the Sobolev completion $\Omega^{0}_{s+1}(\frg_{E})$ of $\Omega^{0}(\frg_{E})$ with respect to the Sobolev norm $L^2_{s+1}$. Hence $T_{\Id}(\cG_{s+1})\cong\Omega^{0}_{s+1}(\frg_{E})$. 	
	\item We define $\cA_{s}$ to be the Sobolev completion of $\cA$ with respect to the Sobolev-norm $L^2_{s}$. Fixing a base (smooth) connection
	$A_0 \in{\cA}_{s}$ we can write:
	$$
	{\cA}_{s} = A_0 + \Omega^{1}_{s}(\frg_{E})\, . 
	$$	
	Using Sobolev completions and taking $s > \frac{1}{2}\dim M$, which will assume henceforth, the action of ${\cG}_{s+1}$ on ${\cA}_{s}$ is smooth.
\end{itemize}

\begin{remark}
The previous remarks show that there are compact, continuous, embeddings of $\cG_{s+1}$ and $\cA_{s}$ respectively into the space of continuous sections $C^{0}(\End E)$ of $\End E$ and the space of continuous sections $C^{0}(T^{\ast}M)$ of $T^{\ast}M$. However, for applications to instantons we may need the previous compact, continuous embeddings to be respectively in $C^{2}(\End E)$ and $C^{2}(T^{\ast}M)$. This can be achieved simply by taking $s$ to be large enough, for example $s>\dim(M)$.
\end{remark}

The curvature operator
$$
\cF_s \colon {\cA}_{s}\to \Omega^{2}_{s-1}(\frg_{E})\, , \qquad \cF_s(A)=F_A\, ,
$$	
extends to a smooth, bounded, $\cG_{s+1}$-equivariant map of infinite-dimensional Hilbert-spaces. There is a natural smooth action of $\cG_{s+1}$ on $\cA_{s}$
 \begin{eqnarray*}
\Phi_{s+1}\colon \cG_{s+1}\times \cA_{s} &\to& \cA_{s}\, ,\\
(u,A) &\mapsto& u\cdot A\, .
\end{eqnarray*}
The centre $Z(\G)$ acts trivially on $\cA_s$, so $\bar{\cG}_s=\cG_s/Z(\G)$ also acts smoothly on $\cA_{s}$. Let us fix a point $A \in {\cA}_{s}$ and define $\Phi^{A}_{s+1} := \Phi_{s+1}(-,A)\colon \cG_{s+1}\to \cA_{s}$. The derivative $(d\Phi^{A}_{s+1})|_{\Id}\colon \Omega^{0}_{s}(\frg_{E})\to \Omega^{1}_{s}(\frg_{E})$ of $\Phi^{A}_{s+1}$ at the identity $\Id\in\cG_{s+1}$ is given by
\begin{eqnarray*}
(d\Phi^{A}_{s+1})|_{\Id}\colon \Omega^{0}_{s+1}(\frg_{E}) &\to & \Omega^{1}_{s}(\frg_{E})\, , \\
\tau &\mapsto & -d_{A}\tau\, . 
\end{eqnarray*}

\begin{definition}
We define the following spaces of connections modulo gauge transformations, equipped with the quotient topology
\begin{equation*}
\cB_{s} = {\cA}_{s}/\cG_{s+1}= \cA_{s}/\bar{\cG}_{s+1}\, . 
\end{equation*}
\end{definition}

\begin{remark}
We denote by $\cO_A : = \cG\cdot A$ the orbit of the $\cG_{s+1}$-action on $\cA_{s}$ passing through $A\in \cA_{s}$. 
The tangent space of $\cO_A$ at $A$ is given by
$$
T_{A}\cO_A = \left\{ d_{A} \gamma \, |\,\, \gamma \in \Omega^{0}_{s+1}(\mathfrak{g}_{E}) \right\}\subset \Omega^{1}_{s}(\mathfrak{g}_{E})\, .
$$
\end{remark}
Let $A\in \cA_{s}$. The stabilizer of $A$ is defined as:
 \begin{equation*}
 \Gamma_{A} = \left\{ u\in{\cG}_{s+1}\,\, | \,\, u\cdot A = A \right\} .
 \end{equation*} 
Elements in $\Gamma_{A}$ correspond to covariantly constant automorphisms of $E$. The Lie algebra of $\Gamma_{A}$ is given by
\begin{equation*}
\mathfrak{t}_A=\text{Lie} \, \Gamma_A = \ker (d_{A}\colon \Omega^{0}_{s+1}(\frg_{E})\to \Omega^{1}_{s}(\frg_{E})) \, .
\end{equation*}

We recall the following well-known lemma.

\begin{lemma}
\label{lemma:GammaH}
For any connection ${{A}}\in{\cA}_{s}$, $\Gamma_{A}$ is isomorphic to the centralizer of the holonomy $H_{A}$ in $\G$. In particular, $\Gamma_{A}$ always contains the center $Z(\G)$ of $\G$. 
\end{lemma}

\begin{remark}
For $G = \SU(r)$ we have $Z(\SU(r)) = \ZZ_{r}$, and for $G=\U(r)$, we have that $Z(\U(r))=\U(1)$, the subgroup of diagonal matrices. 
\end{remark}

\begin{definition}\label{def:reducible}
We say that a connection $A$ is irreducible if the holonomy of $A$ is is equal to the structure group of $P$, i.e., $H_A = \G$. We say it is reducible otherwise.
\end{definition}

If a connection $A$ is reducible, then the holonomy $H_A\subset \G$ is strictly contained in $\G$. Therefore the holonomy Lie algebra $\mathfrak{h}_A\subset \frg$ is strictly contained, and $F_A\in \Omega^2((\mathfrak{h}_A)_{E})$.

\begin{prop} \label{prop:reducible} Let $A\in \cA_{s}$. If $A$ is irreducible, then $\Gamma_A=Z(\G)$ and the kernel of
$d_{{A}}\colon \Omega^{0}_{s+1}(\frg_{E})\to \Omega^{1}_{s}(\frg_{E})$ 
consists of covariantly constant sections. 
\end{prop}

\begin{proof}
If $A$ is irreducible, then $H_A = \G$, so $\Gamma_A$ is the centralizer of $\G$, hence $\Gamma_A=Z(\G)$. 
The Lie algebra of $\Gamma_A$ is $\mathfrak{t}_A$, the kernel of $d_A$, hence $\mathfrak{t}_A$ consists only of the constant sections.
%
\end{proof}

Note that definition \ref{def:reducible} is a more restrictive definition that than used in other instances. 
If there is a splitting $E=E_1\oplus E_2$ with $A=A_1\oplus A_2$, where $E_1$ is a rank $k$ bundle
and $E_2$ is a rank $(r-k)$-bundle, then the holonomy is contained in the
subgroup $\mathrm{S}(\U(k)\x \U(r-k))<\SU(r)$ and hence $A$ is reducible. 
There are situations in which we have a converse statement. For instance, for $M$ simply-connected and $\G=\SU(r),\U(r)$
with $r\leq 3$, if $A$ is reducible, then there is a splitting $E=E_1\oplus E_2$. Certainly, suppose that 
$H_A\neq \G$. By simply-connectivity of $M$, $H_A$ is connected. Being a subgroup of $\U(r)$ with $r\leq 3$, it must be 
conjugated to a subgroup of some $\U(k)\x \U(r-k)$, $0< k < r$. This implies that the bundle and the connection split as indicated. 

\medskip

We denote by $\cA^*_{s} \subset \cA_{s}$ the subspace of irreducible connections on $E$, which is dense and open in $\cA_{s}$.

\begin{definition}
We define the space of irreducible connections modulo gauge transformations, equipped with the quotient topology
\begin{equation*}
\cB^{\ast}_{s} = {\cA}^{\ast}_{s}/\cG_{s+1} = \cA^{\ast}_{s}/\bar{\cG}_{s+1} \subset \cB_s\, .
\end{equation*}
\end{definition}

\begin{remark}
The reduced gauge group $\bar{\cG}_{s + 1}$ acts freely on $\cA^{\ast}_s$. 
\end{remark}

We define the following canonical projections
\begin{equation*}
\pi\colon \cA_{s+1} \to \cB_{s}\, , \qquad \pi \colon \cA^{\ast}_{s+1} \to \cB^{\ast}_{s}\, , 
\end{equation*}
We proceed now to analyze the local structure of $\mathcal{B}_{s}$ following references \cite{AtiyahHitchin,AtiyahBott,DKbook,FriedmanMorgan,BlaineL}.

\begin{lemma}
\label{lemma:dA}
For any $s \geq 0$, the map $d_{A}\colon \Omega^{0}_{s+1}(\frg_{E})\to \Omega^{1}_{s}(\frg_{E})$ has finite-dimensional kernel and closed range. The kernel of $d_{A}$ consists of $C^{\infty}$-sections of $\frg_{E}$ and its dimension satisfies
\begin{equation*}
\dim\, \ker d_{A} \leq \rk \, \frg_{E}\, .
\end{equation*}
Furthermore, there exists a constant $c_{s+1}$ such that
\begin{equation}
\label{eq:cu}
\norm{\psi}_{s+1} \leq c_{s+1} \norm{d_{A}\psi}_{s}\, , 
\end{equation}
for all $\psi \perp \ker d_{A}$.
\end{lemma}

\begin{proof}
Elements $\phi\in \ker d_{A}$ are sections of a vector bundle parallel with respect to the connection $d_{A}$. 
Therefore, they are completely specified by their vale at one point and hence we obtain $\dim\, \ker d_{A} \leq \rk\, \frg_{E}$.
Moreover $(\nabla_A)^{k+1} \phi = 0$ for all $k\geq 0$ and hence $\phi$ is $C^{\infty}$ by the Sobolev embedding theorem. 

Using now the identity $\norm{d_{A}\phi}^{2}_{s} + \norm{\phi}^{2}_{0} = \norm{\phi}^{2}_{s+1}$ for all $\phi\in \Omega^{0}_{s+1}(\frg_{E})$, 
equation \eqref{eq:cu} is equivalent to:
\begin{equation}
\label{eq:equivor}
\norm{\psi}^{2}_{0} \leq c \norm{d_{A}\psi}^{2}_{0}\, ,
\end{equation}
for some constant $c>0$ and for all $\psi \perp \ker\, d_{A}$. We can now prove equation \eqref{eq:equivor} by taking
\begin{equation*}
c^{-1} =\inf \left\{ \left( \frac{\norm{ d_{A}\psi}^{2}_{0}}{\norm{ \psi}^{2}_{0}}\right)\, \, |\, \, \psi \perp \ker\, d_{A}\right\}\, .
\end{equation*}
In order to see that the constant $c^{-1}$ given above is well-defined we just have to consider the following equalities
 $$
  \la\Delta_A \psi,\psi\ra_0 = \la d^{\ast}_{A} d_{A}\psi,\psi \ra_0 
= \la d_{A}\psi, d_{A}\psi\ra_0= \norm{d_{A}\psi}^{2}_{0}\, .
 $$
Hence, $c^{-1}$ is just the first non-zero eigenvalue of $\Delta_A$. Since $\Delta_A$ is elliptic and $M$ is closed, 
the spectrum of $\Delta_A$ is discrete and $\lambda >0$. We conclude that $d_{A}\colon \Omega^{0}_{s+1}(\frg_{E}) 
\to \Omega^{1}_{s}(\frg_{E})$ is bounded from below in the orthogonal of its null space and hence it has closed range. 
Alternatively, the fact that $d_{A}\colon \Omega^{0}_{s+1}(\frg_{E}) \to \Omega^{1}_{s}(\frg_{E})$ has closed range 
follows from the orthogonal decomposition given in equation \eqref{eq:orthsplit}.
\end{proof}

\begin{lemma}
\label{lemma:s1}
Let us assume that two connections $A_{1}, A_{2}\in \cA_{s}$ are equivalent by a gauge transformation 
$u\in \cG_{0}$. Then $u\in\cG_{s+1}$.
\end{lemma}

\begin{proof}
Let us write $A_{a} = A + \tau_{a}$. Equation $A_{1} = u \cdot A_{2}$ implies
 \begin{equation*}
 d_{A} u  = \tau_{2} u - u \tau_{1}\, ,
 \end{equation*}
where the derivatives are understood in the weak sense. Note that $u\in \cG_0$ means that $u$ is $L^2$.
Taking now the $L^{2}$ norm in the above equation,
	\begin{equation*}
\norm{d_{A} u}_{0} \leq \norm{\tau_{2} u}_{0} + \norm{u\tau_{1}}_{0} \leq c \left(\norm{\tau_{2}}_{s} \norm{u}_{0} + \norm{u}_{0} \norm{\tau_{1}}_{s}\right)\, ,
\end{equation*}
where $c>0$ is a positive constant. Here we have used the Sobolev multiplication theorem, which in our particular case gives the appropriate estimate for pointwise multiplication and states that $L^2_{s}\otimes L^2_{i} \to L^2_{i}$ is a continuous bilinear map for $0\leq i \leq s$ provided that $s>\frac{1}{2} \dim(M)$ (which we assume always). 
As $\norm{u}_{0}<\infty$ and $\norm{\tau_{a}}_{s} <\infty$ by assumption, we conclude
\begin{equation*}
\norm{d_{A} u}_{0} <\infty\, .
\end{equation*}
Hence $u \in L^2_{1}$. Iteratively repeating the previous argument, we arrive to the last step
\begin{equation*}
\norm{d_{A} u}_{s} \leq \norm{\tau_{2} u}_{s} + \norm{u\tau_{1}}_{s} \leq c \left(\norm{\tau_{2}}_{s} \norm{u}_{s} + \norm{u}_{s} \norm{\tau_{1}}_{s}\right)\, ,
\end{equation*}
which implies $\norm{d_{A} u}_{s} <\infty$ and hence $u\in \cG_{s+1}$.
\end{proof}

\begin{lemma}
\label{lemma:convergence}
Let $\left\{A^{n}_{1}\right\}$ and $\left\{A^{n}_{2}\right\}$ be sequences of points in $\cA_{s}$ that respectively converge to connections $A_{1}, A_{2}\in \cA_{s}$. Let us assume that $\left\{u_{n}\right\}$ is a sequence in 
$\cG_{0}$ such that
\begin{equation*}
A^{n}_{1} = u_{n}\cdot A^{n}_{2}\, , \qquad \forall\, n\in \mathbb{N}\, .
\end{equation*}
Then $\left\{u_{n}\right\}\subset \cG_{s+1}$. Furthermore, after perhaps passing to a subsequence, $\left\{u_{n}\right\}$ converges to an element $u\in\cG_{s+1}$ satisfying $A_{1} = u\cdot A_{2}$.
\end{lemma}

\begin{proof}
The fact that $u_{n}$ is in $\cG_{s+1}$ for all $n\in\mathbb{N}$ follows directly from lemma \ref{lemma:s1}. We will assume then that  $\left\{u_{n}\right\}\subset \cG_{s+1}$. Let us write $A^{n}_{a} = A + \tau^{n}_{a}$, $a = 1, 2$. The sequences $\left\{\tau^{n}_{a}\right\}$ converge to $\tau_{a}\in \Omega^{1}_{s}(\frg_{E})$ in $L^2_{s}$. Equation $A^{n}_{1} = u_{n}\cdot A^{n}_{2}$ implies
 $$
 d_{A} u_{n} = \tau^{n}_{2} u_{n} - u_{n} \tau^{n}_{1}\, .
 $$
The sequences $\left\{\tau^{n}_{1}\right\}$ and $\left\{\tau^{n}_{2}\right\}$ converge in the $L^2_{s}$-norm, and hence they are uniformly bounded in $\Omega^{1}_{s}(\frg_{E})$. Now, let us consider each $u_{n}$ as a $\G$-equivariant function on $P$ and taking values on $\G\subset \mathrm{Mat}(r,\mathbb{C})$, which is a compact subspace of the vector space of square $r\times r$ complex matrices $\mathrm{Mat}(r,\mathbb{C})$. 
Therefore $\left\{ u_{n}\right\}$ is uniformly bounded in the $L^{\infty}$-norm.
The uniform bound of $\left\{ u_{n}\right\}$ in the $L^{\infty}$-norm implies a uniform bound in the $L^{2}$-norm by compactness of $M$.
The Sobolev multiplication theorem implies now that $L^2_{s}\otimes L^2_{i} \to L^2_{i}$, $i = 0,\hdots, s$, 
is a continuous bilinear map and in addition gives us the following estimate
\begin{equation*}
\norm{d_{A}u_{n}}_{0} = \norm{\tau^{n}_{2} u_{n} - u_{n} \tau^{n}_{1}}_{0}\leq  c \left(\norm{\tau^{n}_{2}}_{s} \norm{u_{n}}_{0} + \norm{u_{n}}_{0} \norm{\tau^{n}_{1}}_{s}\right)\, ,
\end{equation*}
for an appropriate constant $c>0$. Hence, we conclude that $\left\{d_{A}u_{n}\right\}$ is uniformly bounded with respect to $L^{2}$, implying that $\left\{u_{n}\right\}$ is uniformly bounded in $L^2_{1}$. Iteratively repeating this process, we arrive to the last step
\begin{equation*}
\norm{d_{A}u_{n}}_{s} = \norm{\tau^{n}_{2} u_{n} - u_{n} \tau^{n}_{1}}_{s}\leq  c \left(\norm{\tau^{n}_{2}}_{s} \norm{u_{n}}_{s} + \norm{u_{n}}_{s} \norm{\tau^{n}_{1}}_{s}\right)\, ,
\end{equation*}
for a constant $c>0$.  Hence, we conclude that $\left\{d_{A}u_{n}\right\}$ is uniformly bounded with respect to $L^2_{s}$, implying that $\left\{u_{n}\right\}$ is uniformly bounded in $L^2_{s+1}$. Hence, perhaps after passing to a subsequence, $\left\{ u_{n}\right\}$ weakly converges in the $L^2_{s+1}$-norm. Since $\left\{ u_{n}\right\}\subset \cG_{s+1}$ and $\cG_{s+1}$ is a closed subspace of $\Omega^{0}_{s}(\End(E))$ the weak limit $u$ of $\left\{ u_{n}\right\}$ is in $\cG_{s+1}$. We want to prove now that in fact $\left\{ u_{n}\right\}$ strongly converges to $u$. By strict inequality of the Sobolev embedding theorems, we obtain that the embedding
\begin{equation*}
L^{2}_{s+1}\hookrightarrow L^{2}_{s}\, ,
\end{equation*}
is a compact map. Hence, 
$\left\{ u_{n}\right\}$ strongly converges in $L^2_{s}$. Using now that the Sobolev multiplication theorem implies that 
$L^2_{s}\otimes L^2_{s} \to L^2_{s}$ is continuous we conclude that $\left\{ \tau^{n}_{2} u_{n} - u_{n} \tau^{n}_{1}\right\}$ and hence $\left\{ d_{A}u\right\} $ strongly converges in $L^2_{s}$, whence $\left\{ u_{n}\right\}$ strongly converges in 
$L^2_{s+1}$. Furthermore, by uniqueness of limits we obtain
\begin{equation*}
d_{A} u = \tau_{2} u - u \tau_{1}\, .
\end{equation*}  

We have proven that, perhaps after passing to a subsequence, $\left\{u_{n}\right\}$ converges to $u$ in $L^2_{s+1}$ and in particular, $u\in \cG_{s+1}$. Finally, from $d_{A} u = \tau_{2} u - u \tau_{1}$ follows that $A_{1} = u\cdot A_{2}$ and we conclude.
\end{proof}

Let $A\in\cA_{s}$. We define
  \begin{equation}\label{eq:Tnabla}
T_{{{A}},\epsilon} := \left\{ A + \tau\, ,\, \tau \in \Omega^{1}_{s} (\frg_{E}) \,\, |\,\, d^{\ast}_{A} \tau = 0\, , \,\, 
  \norm{\tau}_{s} < \epsilon\right\}\subset \cA_{s}\, .
 \end{equation}

\begin{lemma} \label{lemma:preservestab}
For each $A\in\cA_{s}$, the stabilizer subgroup $\Gamma_{A}\subset\cG_{s+1}$ fixes $T_{A,\epsilon}$.
\end{lemma}

\begin{proof}
Follows from invariance of $A$ under $\Gamma_{A}$ together with the fact that, for all $u\in \Gamma_{A}$, we have $d^{\ast}_{A} (u\cdot \tau) = u\cdot d^{\ast}_{A}\tau$ and
\begin{equation*}
\norm{u\cdot \tau }_{s} = \norm{u \tau u^{-1}}_{s} = \norm{\tau}_{s}<\epsilon.
\end{equation*}
\end{proof}

\begin{lemma}
\label{lemma:A1A2}
Let $A\in \cA_s$. There exists a $\epsilon >0$ such that if $A_{1}, A_{2} \in T_{A,\epsilon}$ are connections equivalent by a gauge transformation $u\in \cG_{s+1}$ with $\norm{u-\Id}_{s+1} <\epsilon $, then $A_{1} = \gamma\cdot A_{2}$ for some $\gamma \in \Gamma_{A} \subset \cG_{s+1}$.
\end{lemma}

\begin{proof}
By lemma \ref{lemma:dA}, the bounded linear map $d_{A}\colon \Omega^{0}_{s+1}(\frg_{E})\to \Omega^{1}_{s}(\frg_{E})$ has closed range. Therefore, there exists the following orthogonal decomposition:
\begin{equation}
\label{eq:TAsplit}
T_{A}\cA_s \cong \Omega^{1}_{s}(\frg_{E}) \cong \im (d_{A}) \oplus \ker (d^{\ast}_{A})\, ,
\end{equation}
where $d^{\ast}_{A}$ is the adjoint of $d_{A}$ with respect to $L^2_{s}$. 
Notice that $\im (d_{A})$ is the tangent space to the orbit $\cO_{A}$ at $A$. We define now the following differentiable map 
of smooth Hilbert manifolds	
\begin{eqnarray*}
\Psi\colon \cG_{s+1}\times T_{A,\epsilon} &\to& \cA_{s}\, , \\
(u, \tau) &\mapsto& u\cdot \tau\, .
\end{eqnarray*}

\noindent	
The differential of $\Psi$ at $(\Id,0)\in \cG_{s+1}\times T_{A,\epsilon}$ is given by:	
\begin{eqnarray*}
d\Psi_{(\Id,0)}\colon \Omega^{0}_{s+1}(\frg_{E})\times \ker\, d^{\ast}_A &\to& \Omega^{1}_{s}(\frg_{E})\, , \\
(l,\tau) &\mapsto& -d_{A}l +\tau\, ,
\end{eqnarray*}
where we have identified $T_{\Id}(\cG_{s+1})\cong \Omega^{0}_{s+1}(\frg_{E})$ and $T_{A}(T_{A,\epsilon})\cong  \ker (d^{\ast}_A)$. 
With respect to the splitting given in equation \eqref{eq:TAsplit} we can write equation $d\Psi_{(\Id,0)}$ as follows	
 $$
 d\Psi_{(\Id,0)} = (-d_{A},\Id)\, ,
 $$	
and hence we conclude that $d\Psi_{(\Id,0)}$ is an isomorphism of Banach spaces if and only if $d_{A}$ is injective. 

For clarity we distinguish now three cases, although strictly speaking each case is a particular case of the next one.	

\begin{itemize}
\item {\bf $A\in\cA^{\ast}_{s}$ irreducible and $Z(\G)=\{\Id\}$.} 
In this case $d_{A}\colon \Omega^{0}_{s+1}(\frg_{E})\to \Omega^{1}_{s}(\frg_{E})$ is injective (it has trivial kernel), 
and hence $d\Psi_{(\Id,0)}$ is an isomorphism. It follows then by the inverse function theorem that there exists a neighborhood		
$$
\cN_{\Id,\epsilon} := \left\{ u\in \cG_{s+1} \,\, |\,\, \norm{u-\Id}_{s+1} <\epsilon \right\}\, ,
$$ 
of $\Id\in \cG_{s+1}$ and a neighborhood $U(A)$ of $A\in \cA^{\ast}_{s}$ such that $\Psi\colon \cN_{\Id,\epsilon}\times T_{A,\epsilon} \to U(A)\subset \cA^{\ast}_{s}$ is a diffeomorphism. Therefore, any connection in $A^{\prime}\in U(A)$ can be written as $A^{\prime} = u\cdot A'_0$ for an appropriate $u\in \cN_{\Id,\epsilon}$ in a neighborhood of the identity
and $A'_0\in T_{A,\epsilon}$. In particular, for the $\epsilon$ appearing in $\cN_{\Id,\epsilon}$, if $A_{1}, A_{2}\in T_{A,\epsilon}$ such that $A_{1} = u\cdot A_{2}$ for a gauge transformation $u\in \cG_{s+1}$ satisfying $\norm{u-\Id}_{s+1} < \epsilon$ then $u=\Id$ and $A_{1} = A_{2}$.
		
\item {\bf $A\in\cA^{\ast}_{s}$ irreducible and any $Z(\G)$.} 
In this case $d_{A}\colon \Omega^{0}_{s+1}(\frg_{E})\to \Omega^{1}_{s}(\frg_{E})$ may have kernel, isomorphic to $\zeta(\frg)$,
the Lie algebra of $Z(G)$. Therefore, we will slightly modify the domain of the map $\Psi$ in order to obtain a diffeomorphism. 
Instead of $\cG_{s+1}\times T_{A,\epsilon}$ we consider
  \begin{equation*}
 (\cG_{s+1}\times T_{A,\epsilon})/Z(\G) \cong \bar{\cG}_{s+1}\times T_{A,\epsilon}\, . 
 \end{equation*}
Here $Z(\G)$ acts trivially on $T_{A,\epsilon}$.
Since $Z(\G)\subset \cG_{s+1}$ is a normal Hilbert subgroup of $\cG_{s+1}$, we have that $\cG_{s+1}/C(\G)$ is again a infinite-dimensional 
Hilbert Lie group with Lie algebra isomorphic to $\Omega^{0}_{s+1}(\frg_{E})/\zeta(\frg)$ (cf.\cite{GNeeb}). 
This isomorphism follows simply from the fact that $Z(\G)$ is the subgroup of $\cG_{s+1}$ consisting on 
parallel sections of the endomorphism bundle with respect to $d_{A}$. We define then the following differentiable map		
\begin{eqnarray*}
\bar{\Psi}\colon \bar{\cG}_{s+1}\times T_{A,\epsilon} &\to& \cA_{s}\, , \\
(u, \tau) &\mapsto& u\cdot \tau\, .
\end{eqnarray*}
Notice that $\bar{\Psi}$ is well-defined since different representatives of an element in $\bar{\cG}_{s+1}$ differ by an element in 
$Z(\G)$ which does not affect $\tau \in \ker\, d^{\ast}_{A}$. The differential $d\bar{\Psi}_{(\Id,0)}$ of $\bar{\Psi}$ at $(\Id,0)$ is given by
\begin{eqnarray*}
d\Psi_{(\Id,0)}\colon \frac{\Omega^{0}_{s+1}(\frg_{E})}{\zeta(\mathfrak{g})}\times \ker\, d^{\ast}_A &\to& \Omega^{1}_{s}(\frg_{E})\, , \\
(l,\tau) &\mapsto& -d_{A}l +\tau\, ,
\end{eqnarray*}
Clearly now $d_{A}$ is inyective and therefore $d\Psi_{(\Id,0)}$ is an isomorphism of infinite-dimensional Banach spaces. 
We conclude that there exists a neighborhood		
 $$
 \bar{\cN}_{\Id,\epsilon} := \left\{ u\in \bar{\cG}_{s+1}\,\, |\,\, \norm{u-\Id}_{s+1} <\epsilon \right\}\, ,
 $$		
of $\Id\in \bar{\cG}_{s+1}$ and a neighborhood $U(A)$ of $A\in \cA^{\ast}_{s}$ such that $\Psi\colon \bar{\cN}_{\Id,\epsilon}\times T_{A,\epsilon} \to U(A)\subset \cA^{\ast}_{s}$ is a diffeomorphism. 

\item {\bf $A\in\cA_{s}$ not necessarily irreducible.} In this case the stabilizer $\Gamma_A$ of $A$ may be non-trivial, and 
$d_{A}$ can have non-trivial kernel, which is isomorphic to the Lie algebra $\mathfrak{t}_A$ of $\Gamma_A$. 
We again slightly modify the definition of $\Psi$ and instead define the following differentiable map
\begin{eqnarray*}
\bar\Psi\colon (\cG_{s+1}\times T_{A,\epsilon}) /\Gamma_{A} &\to& \cA_{s}\, , \\
\left[ u, \tau \right] &\mapsto& u\cdot \tau\, .
\end{eqnarray*}
Here $[u,\tau ]$ denotes the equivalence class of $(u,\tau)$ respect to the action of $\Gamma_{A}$. We first check that $\Psi$ is well-defined, 
namely that its value does not depend on the representative. If $(u,\tau)\in  (\cG_{s+1} \times T_{A,\epsilon})/\Gamma_{A}$ 
then any other representative is of the form 
$(u \gamma^{-1}, \gamma \tau \gamma^{-1})$ for a unique $\gamma \in \Gamma_{A}$. It is then a direct computation to check that 
$\Psi(u,\tau) = \Psi(u \gamma^{-1}, \gamma \tau \gamma^{-1})$. 

The action of $\Gamma_A$ on $\cG_s \x T_{A,\epsilon}$ is free, since it is free in the first variable. 
Therefore $(\cG_{s+1}\times T_{A,\epsilon}) /\Gamma_{A}$
is smooth and the tangent space at $(\Id,0)$ is $ \frac{\Omega^{0}_{s+1}(\frg_{E})}{\mathfrak{t}_{A}}\times \ker d^{\ast}_A$
The differential $d\bar\Psi_{(\Id,0)}$ of $\bar\Psi$ at $(\Id,0)$ is
\begin{eqnarray*}
 d\bar\Psi_{(\Id,0)}\colon \frac{\Omega^{0}_{s+1}(\frg_{E})}{\mathfrak{t}_{A}}\times \ker d^{\ast}_A &\to& \Omega^{1}_{s}(\frg_{E})\, , \\
(l,\tau) &\mapsto& -d_{A}l +\tau\, .
\end{eqnarray*}
Clearly now $d_{A}$ is 
injective and therefore $d\Psi_{(\Id,0)}$ is an isomorphism of infinite-dimensional Banach spaces. We conclude that there exists a neighborhood
 $$
 \bar{\cN}_{\Id,\epsilon} := \left\{ u\in \bar{\cG}_{s+1}\,\, |\,\, \norm{u-\Id}_{s+1} <\epsilon \right\}\, ,
 $$		
of $\Id\in \bar{\cG}_{s+1}$ and a neighborhood $U(A)$ of $A\in \cA^{\ast}_{s}$ such that $\bar\Psi\colon (\bar{\cN}_{\Id,\epsilon}\times T_{A,\epsilon})/
\Gamma_A \to U(A)\subset \cA^{\ast}_{s}$ is a diffeomorphism. 

Therefore, any connection in $A^{\prime}\in U(A)$ can be written as $A^{\prime} = u\cdot A$ for the appropriate $u\in \bar{\cN}_{\Id,\epsilon}$ in a neighborhood of the identity. In particular, if $A_{1}, A_{2} \in T_{A,\epsilon}$ satisfy 
$A_{1} = u\cdot A_{2}$ for a $u\in \bar{\cN}_{\Id,\epsilon}$ then $A_{1} = \gamma\cdot A_{2}$ for a $\gamma\in \Gamma_{A}$.
\end{itemize}
\end{proof}

The previous lemma shows that $T_{A,\epsilon}$ is a local slice for the action of a local neighborhood of the identity in $\cG_{s+1}$. We want to show now that, perhaps after taking a smaller $\epsilon >0$, $T_{A,\epsilon}/\Gamma_{A}$ is a slice for the complete gauge group $\cG_{s+1}$.

\begin{lemma}
\label{lemma:A1A2refined}
Let $A\in \cA_{s}$. There exists a $\epsilon>0$ such that if $A_{1}, A_{2}\in T_{A,\epsilon}$ are equivalent by a gauge transformation $u\in \cG_{s+1}$, then $u = \gamma \in \Gamma_{A}$ and hence $A_{1} = \gamma\cdot A_{2}$.
\end{lemma}

\begin{proof}
Lemma \ref{lemma:A1A2} shows that
\begin{eqnarray*}
\Psi\colon \left(\cG_{s+1}\times T_{A,\epsilon}\right)/\Gamma_{A} &\to& \cA_{s}\, , \\
\left[ u, \tau \right] &\mapsto& u\cdot \tau\, ,
\end{eqnarray*}
is a local diffeomorphism. We want to show that there exists an $\epsilon >0$ such that $\Psi$ is a global diffeomorphism from $\left(\cG_{s+1}\times T_{A,\epsilon}\right)/\Gamma_{A}$ onto its image. If there was not such $\epsilon >0$, there would exist sequences $\left\{A^{n}_{1}\right\}, \left\{A^{n}_{2}\right\} \in T_{A,\epsilon}$ and $\left\{u_{n}\right\}\in \cG_{s+1}$ such that
\begin{equation}
\label{eq:equations}
A^{n}_{1} = u_{n}\cdot A^{n}_{2}\, , \qquad \lim_{n\to\infty} A_{1}^{n} = \lim_{n\to\infty} A_{2}^{n} = A\, , \qquad [A_{1}^{n}] \neq [A_{2}^{n}]\, .
\end{equation}

By lemma \ref{lemma:convergence} we can extract a subsequence of $\left\{u_{n}\right\}$ such that it converges to $u\in \cG_{s+1}$ and satisfies $A = u\cdot A$. Therefore $u\in \Gamma_{A}$. Hence, the sequences $[A^{n}_{1},\Id]$ and $[A^{n}_{2}, u_{n}]$ in $\left(T_{A,\epsilon}\times \cG_{s+1}\right)/\Gamma_{A}$ both converge to $[A,\Id]$. By the local diffeomorphism property of $\Psi$ this implies that for $n> n_{0}$ for some fixed $n_{0}$ we have $[A_{1}^{n},\Id] = [A_{2}^{n}, u_{n}]$, contradicting the third equation in \eqref{eq:equations}.
\end{proof}

From the previous lemma we obtain the following corollary.

\begin{cor}
\label{cor:polyntau}
For every fixed $A\in \cA_{s}$ there exists a polynomial $p(x,y)$ such that if $A_{1} = A + \tau_{1}$ and $A_{2} = A + \tau_{2}$ satisfy $u\cdot A_{2} =A_{1}$ for some $u\in \cG_{s+1}$ then
 $$
 \norm{u}_{s+1} \leq p(\norm{\tau_{1}}_{s},\norm{\tau_{2}}_{s})\, .
 $$
\end{cor}

Lemmas \ref{lemma:A1A2} and \ref{lemma:A1A2refined} finally prove that $T_{A,\epsilon}$ is a slice for the gauge group $\cG_{s+1}$ acting on the space of connections $\cA_{s}$.

\begin{cor}
\label{cor:slice}
For small enough $\epsilon > 0$ and $A\in \cA_{s}$, $T_{A,\epsilon}/\Gamma_{A}$ is a local slice for the action of $\cG_{s+1}$ on $\cA_s$.
\end{cor}

\begin{thm} \label{thm:localorbispace}
The following statements hold:
\begin{itemize}
\item The space $\cB_s$ is a Hausdorff topological space.

\item The subspace $\cB^{\ast}_{s}\subset \cB_{s}$ is an open in $\cB_{s}$.
		
\item The space $\cB^{\ast}_{s}$ is a smooth Hilbert manifold with local charts given by $\pi\colon T_{A,\epsilon}\to\cB^{\ast}_{s}$ for $\epsilon>0$ small enough.	

\item The map $\pi\colon \cA^{\ast}_{s}\to {\cB}^{\ast}_{s}$ is a smooth principal bundle. 

\item For each $A\in \cA^{\ast}_{s}$, the stabilizer $\Gamma_{A}\subset \cG_{s+1}$ preserves $T_{A,\epsilon}$ and the map
\begin{equation*}
h\colon T_{A,\epsilon}/\Gamma_{A}\to \cB_{s}\, ,
\end{equation*}
is a homeomorphism onto a neighborhood of $h([A])$, which in addition is a diffeomorphism outside the fixed point set of $\Gamma_{A}$.
\end{itemize}
\end{thm}

\begin{proof}
The fact that the projection $\pi\colon \cA_{s}\to \cB_{s}$ is open when $\cB_{s}$ is equipped with the quotient topology together with the fact that $\cA_{s}$ is first countable imply that $\cB_{s}$ is also first countable. Hence $\cB_{s}$ is Hausdorff if and only if every convergent sequence has a unique limit. Let us assume then that there exists a convergent sequence with two different limits. In other words, we assume that there exist sequences $\left\{ A^{n}_{1}\right\}$ and $\left\{ A^{n}_{2}\right\}$ of connections such that for all $n\in \mathbb{N}$ we have
\begin{equation*}
A^{n}_{1} = u_{n}\cdot A^{n}_{2}\, , \qquad u_{n}\in \cG_{s+1}\, ,
\end{equation*}
and such that $A_{1} = \lim_{n\to\infty} A^{n}_{1}$ and $A_{2} = \lim_{n\to\infty} A^{n}_{2}$ are gauge-inequivalent connections. We set $A_{2} = A_{1} + \tau_{2}$, $A^{n}_{2} = A_{1} + \tau^{n}_{2}$ and $A^{n}_{1} = A_{1} + \tau^{n}_{1}$, where $\tau_{2}, \tau^{n}_{2}, \tau^{n}_{1} \in \Omega^{1}_{s}(\frg_{E})$ for all $n\in \mathbb{N}$. By hypotesis we have $A_{1} = \lim_{n\to\infty} A^{n}_{1}$ and hence $\lim_{n\to \infty} \tau^{n}_{1} = 0$
and $\lim_{n\to \infty} \tau^{n}_{2} = \tau_2$, in the Sobolev $L^2_{s}$-norm. In particular $\norm{\tau^{n}_{1}}_{s}$ and $\norm{\tau^{n}_{2}}_{s}$ are uniformly bounded, which implies, using corollary \ref{cor:polyntau}, that there exists a constant $c>0$ such that $\norm{u_{n}}_{s+1} < c$ for all $n\in\mathbb{N}$. Applying now the compact Sobolev embedding theorem we conclude that there is a subsequence of $\left\{u_{n}\right\}$ which converges strongly to an element $u\in \cG_{s+1}$ in the Sobolev $L^2_{s}$-norm. Using now that
\begin{equation*}
d_{A} u_{n} = u_{n} \tau^{n}_{1} - \tau^{n}_{2} u_{n}\, , \qquad \forall n\in \mathbb{N}\, ,
\end{equation*}
we obtain
\begin{equation*}
\norm{d_{A_{1}} u_{n}}_{s} \leq c \left( \norm{u_{n}}_{s} \norm{\tau^{n}_{1}}_{s} + \norm{\tau^{n}_{2}}_{s} \norm{u_{n}}_{s}\right)\, ,
\end{equation*}
and hence we conclude that $\left\{ u_{n}\right\}$ converges to $u\in \cG_{s+1}$ in the $L^2_{s+1}$-norm. Therefore, $A_{1} = u\cdot A_{2}$ and every convergent sequence in $\cB_{s}$ has a unique limit, which in turn proves that $\cB_{s}$ is Hausdorff.

In order to prove that $\cB^{\ast}_{s}\subset \cB_{s}$ is open we prove that for every $[A]\in \cB^{\ast}_{s}$ there exists a neighborhood 
$U([A])\subset \cB^{\ast}_{s}$ of $[A]$ contained in $\cB^{\ast}_{s}$. Let us assume otherwise. Then, every neighborhood of 
$[A]$ contains a reducible connection and, since $\cB_{s}$ is first countable, 
this implies the existence of a sequence $\left\{[A^{R}_{n}]\right\}$ of reducible connections in $\cB_{s}$ converging to $[A]\in \cB^{\ast}_{s}$.
The definition of reducibility means that the holonomy Lie algebra $\frh_{A^R_n}$ is not the total space. Passing to a subsequence, we may
assume that $\frh_{A^R_n}$ tends to a subspace $\frh \subset \frg$. So $F_A =\lim_{n\to \infty} F_{A^R_n} \in \Omega^2(\frh)$. We take
$s>0$ large enough so that the connections are $C^1$, and hence $F_A$ is $C^0$. This implies that $A$ has to be reducible.

Corollary \ref{cor:slice} implies that every $A\in {\cB}^{\ast}_{s}$ has a neighborhood homeomorphic to $T_{A,\epsilon}$, which is an infinite-dimensional Hilbert space. The fact that on overlapping open sets in ${\cB}^{\ast}_{s}$ the corresponding changes of coordinates are smooth, as well as the fact that $\pi: \cA^{\ast}\to {\cB}^{\ast}_{s}$ is a smooth principal bundle are both proven in great detail in \cite{MitterViallet}. The rest of the statements follow now from lemma \ref{lemma:preservestab} and corollary \ref{cor:slice}.
\end{proof}

\begin{remark}
The principal fibration
$$
\bar{\cG}_{s+1} \to \cA^{\ast}_{s}\to {\cB}^{\ast}_{s}
$$
induces a long exact sequence in homotopy which implies that
 $$
\pi_{k+1}({\cB}^{\ast}_{s}) = \pi_{k}(\bar{\cG}_{s+1})\, , \qquad k \geq 0\, .
 $$
In particular $\pi_{1}({\mathcal{B}}^{\ast}_{s}) = \pi_{0}(\bar{\cG}_{s+1} )$. 
\end{remark}

\section{Local analysis of the moduli space of $\Spin(7)$-instantons} \label{sec:moduli-local}

As in the previous section, let  $(M,\Omega)$ be an $8$-dimensional manifold with an $\Spin(7)$-structure $\Omega$ and let $P$ be a principal 
$\G$-bundle over $M$, where $\G$ is a compact semi-simple Lie group. Associated to $P$ we consider a complex vector bundle 
$E = P\times_{\rho} \mathbb{C}^{r}$ of rank $r$, where $\rho$ denotes an $r$-dimensional faithful irreducible complex representation of $\G$. 
As explained in Section \ref{sec:Spin(7)rep}, there is a decomposition $\Lambda^2=\Lambda^2_7\oplus \Lambda^2_{21}$ and projections 
$\pi_7:\Lambda^2\to \Lambda^2_7$ and $\pi_{21}:\Lambda^2\to \Lambda^2_{21}$. The following is the central object of this paper.

\begin{definition}\cite{Tian} \label{def:Spinstanton}
A conneciton $A\in \cA_{s}$ is a $\Spin(7)$-instanton if $\pi_{7}(F_A) = 0$. 
\end{definition}

\begin{remark}
The $\Spin(7)$-instanton equation $\pi_7(F_A)=0$ is equivalent to
 \begin{equation}\label{eq:instantoncondition}
\ast F_A = - \Omega\wedge F_A\, ,
 \end{equation}
which is a first-order equation on the connection $A\in\cA$. It implies a second-order equations on $A$.  Acting with $d_{A}$ on \eqref{eq:instantoncondition} we obtain, using the Bianchi identity $d_A F_A=0$, 
\begin{equation} 
\label{eq:instanton2nd}
d^{\ast}_{A} F_A = - \ast (F_A\wedge W)\, ,
\end{equation}
where $W = d\Omega$ is the torsion of the $\Spin(7)$-structure. Interestingly enough, equation \eqref{eq:instanton2nd} follows from the following action, 
which generalizes the standard Yang-Mills action of a connection
 \begin{equation} \label{eq:action}
S(A)= \int_{M} \left( \kappa\left(F_A\wedge\ast F_A\right) + \kappa \left(F_A\wedge F_A\right) \wedge \Omega \right)\, .
\end{equation}
Here $\kappa$ is the bilinear form induced on the adjoint bundle $\mathfrak{g}_{E}$ by the Killing form of $\mathfrak{g}$. For a $\Spin(7)$-holonomy manifold $(M,\Omega)$, the second term in \eqref{eq:action} is topological and we (classically) obtain the standard Yang-Mills action. For a general $\Spin(7)$-manifold $(M,\Omega)$, not necessarily of $\Spin(7)$-holonomy, this term is not topological and does indeed contribute to the equations of motion. To the knowledge of the authors, the physical interpretation of \eqref{eq:action} is still open. For example, it would be interesting to see if it can be supersymmetrized or obtained by dimensional reduction from a supersymmetric Yang-Mills theory on a curved ten-dimensional background.
\end{remark}

\begin{remark}\label{rem:spinorial-Spin7}
 The $\Spin(7)$-instanton equation can be written in terms of the spinor $\eta$ defining the $\Spin(7)$-structure on $M$.
The condition $\pi_7(F_A)=0$ is equivalent to
\begin{equation} \label{eqn:spinorial-Spin7}
	F_{A}\cdot \eta = 0\, .
\end{equation}
\end{remark}

We are interested in studying the \emph{moduli space} of $\Spin(7)$-instantons on $E$, namely the space of connections in $\cA$ satisfying \eqref{eq:instantoncondition} modulo gauge transformations. We define the moduli space of $\Spin(7)$-instantons as follows
\begin{equation*}
\frM = \left\{{{A}}\in{\cA}\,| \, \pi_{7}(F_A) = 0\right\}/\cG\, .
\end{equation*}

As we are working with Sobolev norms to have more control of the topologies involved, we introduce the spaces:
\begin{equation*}
\frM_{s} = \left\{{{A}}\in{\cA}_{s}\,| \, \pi_{7}(F_A) = 0\right\}/{\cG}_{s+1} = \left\{{{A}}\in{\cA}_{s}\,| \, \pi_{7}(F_A) = 0\right\}/\bar{\cG}_{s+1}\, ,
\end{equation*}
as well as the subspace of irreducible $\Spin(7)$-instantons
\begin{equation*}
\frM^*_{s} = \frM_{s} \cap \cB^*_{s} = \left\{{{A}}\in{\cA}^*_{s}\,\, |\,\, \pi_{7}(F_{A}) = 0\right\}/\cG_{s+1} 
= \frM_{s} \cap \bar{\cB}^*_{s} = \left\{{{A}}\in{\cA}^*_{s}\,\, |\,\, \pi_{7}(F_{A}) = 0\right\}/\bar\cG_{s+1}\,
\end{equation*}

\begin{remark}
We equip $\frM_{s}$ and $\frM^{\ast}_{s}$ with the subspace topologies, inherited from $\cB_{s}$.
\end{remark}

Under suitable conditions, we will obtain that $\bar{\frM}^{\ast}_{s}$ is a smooth finite-dimensional manifold.

\begin{definition}
We define the following linear operator of infinite-dimensional Hilbert spaces
\begin{eqnarray} 
\label{eq:Lnabla}
L_{A}\colon \Omega^{1}_s(\frg_{E}) &\to & \Omega^{0}_{s-1}(\frg_{E}) \oplus \Omega^{2}_{7,s-1}(\frg_{E})\, , \nonumber\\
\tau &\mapsto & d_{A}^{\ast}\tau\oplus \pi_{7}\left(d_{A}\tau\right)\, .
\end{eqnarray}
\end{definition}

\begin{lemma}
\label{lemma:LAelliptic}
The linear operator (\ref{eq:Lnabla}) is elliptic.
\end{lemma}

\begin{proof}
By proposition \ref{prop:DiracLA} the symbol of $L_{A}$ is equal to the symbol of the Dirac operator associated to the Ivanov connection of the underlying $\Spin(7)$-structure coupled to the connection induced by $A$ on the endomorphism bundle. Hence the result follows.
\end{proof}

\begin{prop} \label{prop:Cinfty}
Let $A\in{\cA}_{s}$ be a $\Spin(7)$-instanton. Then $A$ is gauge equivalent to a smooth connection. 
Furthermore, for any $s > \frac{1}{2} \dim(M)$ we have the identification $\frM\cong \frM_s$ and 
$\frM$ naturally becomes a second-countable, Hausdorff, metrizable topological space. 
\end{prop}

\begin{proof}
First we prove that if $A_{I}\in \cA_{s}$ is a $\Spin(7)$-instanton, then there exists a gauge transformation $u\in \cG_{s+1}$ such that $u\cdot A\in \cA$, that is, $u\cdot A$ is smooth. 
Let then $A_{I}$ be a $\Spin(7)$-instanton. Since smooth connections are dense in $\cA_{s}$ when equipped with the $L^2_{s}$-topology, for every $\delta > 0$ there exists a smooth connection $A_{S}$ such that $\norm{A_{I}-A_{S}}_{s} < \delta$. We apply now theorem \ref{thm:localorbispace} to $A_{S}$, obtaining the existence of a $\epsilon_{A_{S}} > 0$ such that for any connection $A \in \cA_{s}$ satisfying $\norm{A-A_{S}}_{s} < \epsilon_{A_{S}} $ there exists a unique gauge transformation $u\in\cG_{s+1}$ such that
\begin{equation*}
u\cdot A = A_{S} + \tau\, , \qquad d^{\ast}_{A}\tau = 0 \, ,\qquad  \norm{\tau}_{s} < \epsilon_{A_{S}}\, .
\end{equation*}
We will take $\epsilon_{A_{S}}$ to be the supremum positive real number for which theorem \ref{thm:localorbispace} applies. We claim now that for every $A\in \cA_{s}$, and in particular for $A_{I}$, there exists a smooth connection $A_{S}$ such that $\norm{A-A_{S}}_{s}<\epsilon_{A_{S}}$. To prove this we use the fact that $\cA_{s}$ is a metric space and in particular Fr\'echet-Urysohn. Therefore, there exists a sequence $\left\{A^{n}_{S}\right\}$ of smooth connections converging to $A$ in the $L^2_{s}$-topology, i.e.
\begin{equation*}
\lim_{n\to\infty}\norm{A^{n}_{S} - A}_{s} = 0\, . 
\end{equation*} 
In addition, the sequence $\left\{A^{n}_{S}\right\}$ implies the existence of a sequence of positive numbers $\left\{\epsilon_{A^{n}_{S}}\right\}\subset \mathbb{R}^{+}$. We must have now $\epsilon_{A^{n}_{S}} > \norm{A^{n}_{S}-A}_{s}$ for at least one $n = n_{0}\in \mathbb{N}$. Otherwise $\epsilon_{A^{n}_{S}} \leq \norm{A^{n}-A}_{s}$ for all $n$ and hence  $\lim_{n\to\infty} \epsilon_{A^{n}_{S}} = 0$, which implies, since $\left\{A^{n}_{S}\right\}$ converges to $A$,  that there is no $\epsilon > 0$ satisfying theorem \ref{thm:localorbispace} when applied to $A$, and hence there is no slice $T_{A,\epsilon}$ around $A$, whence running into a contradiction with theorem theorem \ref{thm:localorbispace} and hence proving the initial claim, since $\epsilon_{A^{n_{0}}_{S}} >\norm{A^{n_{0}}_{S}-A}_{s}$ for at least one $n_{0}\in\mathbb{N}$. 

Let then $A_{S}$ be a smooth connection such that $\norm{A_{I}-A_{S}}<\epsilon_{A_{S}}$. By theorem \ref{thm:localorbispace}, there exists a unique gauge transformation $u\in\cG_{s+1}$ such that
\begin{equation*}
u\cdot A_{I} = A_{S} + \tau\, , \qquad d^{\ast}_{A_{S}}\tau = 0 \, ,\qquad  \norm{\tau}_{s} < \epsilon_{A_{S}}\, .
\end{equation*}
Since $A_{I}$ is a $\Spin(7)$-instanton and $F_{A_I}=F_{A_S}+ d_{A_{S}} \tau +\frac{1}{2} [\tau, \tau]$, 
it follows that $\tau$ satisfies the following equation
\begin{equation*}
\pi_{7}(F_{A_{S}}) + \pi_{7} \left( d_{A_{S}} \tau +\frac{1}{2} [\tau, \tau]\right) = 0 \, .
\end{equation*}
We can rewrite the previous equation together with gauge-fixing condition as follows
\begin{equation*}
L_{A_{S}}(\tau) + \left(0, \frac{1}{2} [\tau, \tau]\right) = (0,- \pi_{7}( F_{A_{S}})) \, ,
\end{equation*}
where $L_{A_{S}}\colon \Omega^{1}_s(\frg_{E})\to  \Omega^{0}_{s-1}(\frg_{E}) \oplus \Omega^{2}_{7,s-1}(\frg_{E})$ 
is the linear operator of infinite-dimensional Hilbert spaces defined in equation \eqref{eq:Lnabla}, which by lemma \ref{lemma:LAelliptic} is elliptic. 
Clearly $F_{A_{S}}$ is smooth and by the Sobolev multiplication theorem the term $\frac{1}{2} [\tau, \tau]\in \Omega^{2}_{s}(\frg_{E})$ is in $L^2_{s}$. 
Therefore, applying the regularity theorem for elliptic operators on Sobolev spaces to the equation above we conclude that 
$\tau\in \Omega^{1}_{s+1}(\frg_{E})$. Repeating the argument, $\tau\in \Omega^1_{s+k}(\frg_E)$, for all $k>0$, and hence $\tau$ is smooth.
This implies that $u\cdot A_I$ is a smooth connection.

Let us now consider the map 
\begin{equation*}
\mathfrak{i}\colon\frM \to \frM_{s}\, ,
\end{equation*}
where the topology of $\frM$ is induced by the Fr\'echet topolgoy of $\cA$ given by the $C^\infty$-convergence.
In particular, $\mathfrak{i}$ is continuous. The previous argument shows that the map is surjective. We see that it is injective as follows:
let $A_1,A_2$ be two smooth $\Spin(7)$-instantons, and suppose that $\mathfrak{i}([A_1])=\mathfrak{i}([A_2])$ in $\cB_s$. Then there exists $u\in \cG_{s+1}$ such
that $A_2=u\cdot A_1$. By Lemma \ref{lemma:s1}, we have that $u\in \cG_{s+k}$ for all $k>0$. Hence $u$ is smooth and $[A_1]=[A_2]$ in $\frM$.
Finally, let us see that $\mathfrak{i}$ is a closed map. Suppose that there is a sequence $[A_n]$ in $\frM$ and $[A]\in \frM$ such that $\mathfrak{i}([A_n])\to \mathfrak{i}([A])$,
we have to see that $[A_n]\to [A]$ in $\frM$.
Consider the slice $T_{A,\epsilon}$, then there exist some $u_n\in\cG_{s+1}$ such
that $u_n\cdot A_n\to A$ in $L^2_s$, $u_n\cdot A_n=A+\tau_n$, $d_A^*\tau_n=0$, and $\norm{\tau_n}_s\to 0$.
Therefore, as before 
\begin{equation*}
L_{A}(\tau_n) + \left(0, \frac{1}{2} [\tau_n, \tau_n]\right) = (0, 0)\, ,
\end{equation*}
since $\pi_7(F_A)=0$. Then inductively, $\tau_n\in \Omega^1_{s+k}(\frg_E)$ and so $\tau_n$ is $C^\infty$. 
Now decompose $\tau_n=\tau_n^0+\tau_n^\perp$ according to the decomposition $\ker d_A^*=\HH^1_A \x \ker (\pi_7\circ d_A)$.
Let $c>0$ be the first non-zero eigenvalue of $L_A$. Then 
 $$
 \norm{\tau_n^\perp}_{s+k+1} \leq c ^{-1} \norm{L_A(\tau_n^\perp)}_{s+k} =c ^{-1} \norm{L_A(\tau_n)}_{s+k} \leq C\norm{\tau_n}_{s+k}^2, 
 $$ 
where $C>0$ is a constant independent of $s+k$. 
Moreover all norms $\norm{\text{-}}_s$ on $\HH^1_A$ are equivalent since this is a finite-dimensional vector space. 
So if $\norm{\tau_n}_s\to 0$ then $\norm{\tau_n^0}_{s+k}\to 0$. Also by induction on $k>0$ and the above inequality,
$\norm{\tau_n^\perp}_{s+k} \to 0$. Then $\norm{\tau_n}_{s+k}\to 0$, for all $k>0$. Thus $\tau_n\to 0$ in the $C^\infty$-topology
and $[A_n]\to [A]$ in $\frM$.

\end{proof}

We define now the $\pi_{7}$-projection of the curvature operator
 $$
 \cF_{7,s} := \pi_{7}\circ \cF_{s} \colon {\cA}_{s}\to \Omega^{2}_{7,s-1}(\frg_{E})\, ,
$$
which is a smooth, bounded, $\cG_{s+1}$-equivariant map of infinite-dimensional Hilbert-manifolds. 
The preimage of zero under $\cF_{7,s}$ is the space of $\Spin(7)$-instantons on $E$. We clearly have
\begin{equation*}
 \frM_{s} = \left(\cF_{7,s}\right)^{-1}(0)/\bar\cG_{s+1}\, .
\end{equation*}
For simplicity we define, for every $\tau\in \Omega^{1}(\frg_{E})$, $\Psi_{A}(\tau) := \cF_{7,s}(A+\tau)$. Hence
\begin{eqnarray*} 
\Psi_{A} \colon \Omega^{1}_{s}(\frg_{E}) &\to&  \Omega^{2}_{7,s-1}(\frg_{E})\, , \nonumber\\
\tau &\mapsto& \pi_{7}\left(d_{A}\tau + \frac{1}{2}[\tau, \tau]\right)\, .
\end{eqnarray*}

We are interested in characterizing the local geometry of $\frM_s$. Consider $T_{A,\epsilon}$ as given in \eqref{eq:Tnabla}. We define the following restriction of $\Psi_{A}$
\begin{equation*}
\Psi_{A,\epsilon} := \Psi_{A}|_{T_{{{A}}, \epsilon}}\colon T_{{{A}}, \epsilon}\to \Omega^{2}_{7,s-1}(\frg_{E})\, .
\end{equation*}
In addition we define
\begin{equation*}
Z(\Psi_{A,\epsilon}) := \Psi^{-1}_{{{A}}, \epsilon}(0) \subset T_{{{A}}, \epsilon}\, .
\end{equation*}

\begin{remark}
If $A\in \cA^{\ast}$, then there exists $\epsilon >0$ such that $T_{A.\epsilon}\subset \cA^{\ast}$. In that case $Z(\Psi_{A,\epsilon})\subset \frM^{\ast}_{s}$.
\end{remark}

\begin{remark}
The space $Z(\Psi_{A,\epsilon})$ can be explicitly defined as
\begin{equation*}
Z(\Psi_{A,\epsilon}) := \left\{ A + \tau\, ,\, \tau \in \Omega^{1}_{s} (\frg_{E}) \,\, |\,\, d^{\ast}_{A} \tau = 0\, , \,\,  \Psi_{A}(\tau) = 0\, , \,\, \norm{\tau}_{s} < \epsilon\right\}\, .
\end{equation*}
Since $\Psi_{A}$ is not a linear map, $Z(\Psi_{A,\epsilon})$ is a closed subset of $T_{A,\epsilon}$ which is not a linear subspace.
\end{remark}

The following proposition follows from theorem \ref{thm:localorbispace} by restricting the homeomorphism $h$ using the instanton condition.

\begin{prop}
For sufficiently small $\epsilon>0$, the homeomorphism $h\colon T_{{{A}},\epsilon}/\Gamma_{A} \to U([{{A}}])\subset \cB_{s}$ introduced 
in theorem \ref{thm:localorbispace} induces a homeomorphism
\begin{equation*}
h\colon Z(\Psi_{A,\epsilon})/\Gamma_{A} \to U([{{A}}])\cap \frM_{s} \subset \frM_{s}\, .
\end{equation*}
\end{prop}

In particular, ${\frM}^{\ast}$ is locally homeomorphic to $Z(\Psi_{A,\epsilon})$ for every $A\in \cA^{\ast}_{s}$. 

In order to proceed further we need to examine in more detail the zero set $Z(\Psi_{A,\epsilon})$ and the local map $\Psi_{A,\epsilon}$. 
In particular, we will show that $\Psi_{{{A}}, \epsilon}$ is Fredholm on $T_{A,\epsilon}$, a fact which will give us a local model for 
the moduli space of $\Spin(7)$-instantons in terms of the appropriate cohomology groups.

Let us recall that by definition $\Psi_{{{A}},\epsilon}$ is Fredholm on $T_{{{A}},\epsilon}$ if and only if, at every point in $\tau_0\in T_{{{A}},\epsilon}$, the derivative $D_{\tau_0}\Psi_{A,\epsilon}$ is a Fredholm linear operator of Hilbert spaces. In our case, this derivative is independent of $\tau_0$ so we will drop the subscript. It is given by
 \begin{eqnarray} 
\label{eq:linearpsi}
D\left(\Psi_{{{A}},\epsilon}\right)\colon \ker d^{\ast}_{A} &\to& \Omega^{2}_{7,s-1}(\frg_{E})\, ,\nonumber\\
\tau &\mapsto& \pi_{7}\left(d_{A} \tau\right)\, .
\end{eqnarray}

\begin{lemma}
\label{lemma:Delliptic}
Let $A\in \cA$ be a $\Spin(7)$-instanton. Then $D(\Psi_{A,\epsilon})\colon \ker\, d^{\ast}_{A} \to \Omega^{2}_{7,s-1}(\frg_{E})$ is a Fredholm linear operator.
\end{lemma}

\begin{proof}
Let us consider the sequence
\begin{equation}
\label{eq:complex}
0\to \Omega^{0}_{s+1}(\frg_{E}) \xrightarrow{d_{A}} \Omega^{1}_{s}(\frg_{E}) \xrightarrow{\pi_{7}\circ d_{A}} \Omega^{2}_{7,s-1}(\frg_E)\to 0\, ,
\end{equation}
which in addition is a complex since $A$ is an $\Spin(7)$-instanton. The associated symbol complex is a complex of vector bundles and linear maps given by
\begin{equation}
\label{eq:symbolc}
0\to p^{\ast}(\Lambda^{0}(M)\otimes\frg_{E}) \xrightarrow{\delta_{1}} p^{\ast}(\Lambda^{1}(M)\otimes\frg_{E}) \xrightarrow{\delta_{2}} p^{\ast}(\Lambda^{2}(M)\otimes\frg_{E})\to 0\, ,
\end{equation}
where we denote by $p\colon T^{\ast}M\to M$ the corresponding projection. The linear maps $\delta_{1}$ and $\delta_{2}$ evaluated at the cotangent vector $x\in T^{\ast}M$ are given by
\begin{equation*}
\delta_{1}|_{x} = (\wedge x)\otimes \Id\, , \qquad \delta_{2}|_{x} = \pi_{7} \circ ((\wedge x)\otimes\Id)\, .
\end{equation*}

Using the previous expressions for $\delta_{1}$ and $\delta_{2}$, exactness of \eqref{eq:symbolc} 
follows by direct computation. Therefore the images of the complex \eqref{eq:complex} are closed 
subspaces of the corresponding Hilbert spaces and the associated cohomology groups are finite-dimensional. In addition we obtain
\begin{equation}
\label{eq:orthsplit}
\Omega^{1}_{s}(\frg_{E}) = d_{A}(\Omega^{0}_{s+1}(\frg_{E}))\oplus \ker d^{\ast}_{A}\, 
\end{equation}
which implies that the kernel and cokernel of \eqref{eq:linearpsi} respectively correspond to the first and second cohomology groups of the complex \eqref{eq:complex}. Since \eqref{eq:complex} is elliptic they are finite-dimensional. 
\end{proof}

\begin{remark}
Lemma \ref{lemma:Delliptic} also follows from ellipticity of the linear operator
\begin{eqnarray*} 
L_{A}\colon \Omega^{1}_s(\frg_{E}) &\to & \Omega^{0}_{s-1}(\frg_{E}) \oplus \Omega^{2}_{7,s-1}(\frg_{E})\, , \nonumber\\
\tau &\mapsto & d_{A}^{\ast}\tau\oplus \pi_{7}\left(d_{A}\tau\right)\, ,
\end{eqnarray*}
defined in \eqref{eq:Lnabla}, see lemma \ref{lemma:LAelliptic}.
\end{remark}

\begin{remark}
Let us rewrite the deformation of the $\Spin(7)$-instanton equation in terms of spinors given in (\ref{eqn:spinorial-Spin7}). 
Suppose that $A\in \cA$ satisfies $F_A\cdot\eta=0$, and let $A^{\prime} = A + \tau$ with $\tau\in \Omega^{1}(\mathfrak{g}_{E})$.
Then $F_{A'}\cdot \eta=0$ is equivalent to $( d_{A}\tau + \frac{1}{2} [\tau, \tau ])\cdot \eta = 0$.
The linearization of this equation is given by
 $$
 d_{A}\tau \cdot \eta = 0\, ,
 $$
which has to be supplemented with the gauge-fixing condition $d_{A}^\ast \tau = 0$. 

Let $\nabla^{T}$ be the spin Ivanov connection associated to the $\Spin(7)$-structure, which is the unique connection with fully antisymmetric torsion $T$ such that $\nabla^T\eta=0$. We follow now the idea proposed in \cite{Charbonneau} applied to our particular case. Associated to the Ivanov connection $\nabla^{T}$ on the spinor bundle and to the connection induced by $A$ 
on $\frg_{E}$ we consider the corresponding Dirac operator $D^{-}_{T}\colon\Omega^{0}(S^{-}\otimes\frg_{E})\to \Omega^{0}(S^{+}\otimes\frg_{E})$. 
Then for $\tau\in \Omega^{1}(\mathfrak{g}_{E})$ we have that
 \begin{equation}\label{eq:definspin}
 d_{A}\tau \cdot \eta = 0, \, d_{A}^{\ast} \tau = 0\, \quad \text{if and only if} \quad D^{-}_{T}(\tau\cdot\eta) = 0\, ,
 \end{equation}
as it happens in \cite{Charbonneau}. We define the following map
\begin{eqnarray*}
Q_{A}\colon \Omega^{1}(\mathfrak{g}_{E})  &\to & \Omega^{0}(\mathfrak{g}_{E}) \oplus \Omega^{0}(H \otimes\mathfrak{g}_{E})\, , \nonumber\\
\tau &\mapsto & d_{A}^{\ast}\tau\oplus d_{A}\tau\cdot\eta\, ,
\end{eqnarray*}
where $S^+=\la\eta\ra\oplus H$ orthogonally. Thus $Q_{A}(\tau) = 0$ if and only if equations \eqref{eq:definspin} hold.
Using the isomorphism (\ref{eqn:cI}), we have $\Omega^{0}(H \otimes\mathfrak{g}_{E}) \cong \Omega^{2}_{7}(\mathfrak{g}_{E})$, and
we can write $Q_{A}$ as
\begin{eqnarray*}
Q_{A}\colon \Omega^{1} (\mathfrak{g}_{E})  &\to & \Omega^{0}(\mathfrak{g}_{E}) \oplus \Omega^{2}_{7}(\mathfrak{g}_{E})\, , \nonumber\\
\tau &\mapsto & d_{A}^{\ast}\tau\oplus \pi_{7}\left(d_{A}\tau\right)\, ,
\end{eqnarray*}
which coincides with $L_{A}$ defined in equation \eqref{eq:Lnabla}. So, as expected, the infinitesimal deformations of the spinorial 
$\Spin(7)$-instanton condition (\ref{eqn:spinorial-Spin7}) are equivalent to the infinitesimal deformations of the $\Spin(7)$-instanton equation.
\end{remark}

The cohomology groups of the complex \eqref{eq:complex} are given by
\begin{eqnarray*}
\label{eq:hyperc}
\HH^{0}_{A} &=& \ker d_{A} = \coker  L_{A} \cap \Omega^{0}_{s+1}(\frg_{E})\, ,\nonumber \\
\HH^{1}_{A} &=& \frac{\ker (\pi_{7}\circ d_{A})}{\img(d_{A})} = \ker  L_{A} \cap \Omega^{1}_s(\frg_{E})\, , \\
\HH^{2}_{A} &=& \frac{\Omega^{2}_{7,s-1}(\frg_{E})}{\img (\pi_{7}\circ d_{A})} = \coker L_{A} \cap \Omega^{2}_{s-1}(\frg_{E})\, .\nonumber 
\end{eqnarray*}

Here $\HH^{0}_{A}$ is the space of infinitesimal automorphisms of $A$, $\HH^{1}_{A}$ is the space of infinitesimal deformations of the $\Spin(7)$-instanton $A$ and $\HH^{2}_{A}$ is the space of infinitesimal obstructions. If a connection ${{A}}$ is irreducible then $\HH^{0}_{A} = 0$. We say that $A$ is \emph{regular} if $\HH^{2}_{A} = 0$. 

\begin{prop} \label{prop:index}
The index of $L_{A}$ is given by 
\begin{equation*}
\index(L_A)= - (\dim \frg) \hat{A}_2(M) + \frac{1}{24} \la p_1(M), p_1(\frg_E)\ra -\frac{1}{12} (p_1(\frg_E)^2-2p_2(\frg_E))\, .
\end{equation*}
where $\hat{A}_2(M)=\frac{7p_1(M)^2-4p_2(M)}{5670}$ is the second term of the $\hat{A}$-genus of $M$.
\end{prop}

\begin{proof}
Proposition \ref{prop:DiracLA} shows that
\begin{equation*}
\index(L_A) = \index(D^{-}_{T}) = \index(D^{-})\, ,
\end{equation*}
where $D^{-}\colon \Omega^{0}(S^{-}\otimes\frg_{E}) \to \Omega^{0}(S^{+}\otimes \frg_{E})$ is the Dirac operator associated to the Levi-Cevita spin connection and the connection induced by $A$ in $\frg_{E}$. In eight dimensions the chiral complex spin bundles $S^{\pm}_{\mathbb{C}}$ are the complexifications of the real chiral spin bundles $S^{\pm}$. Let $D^{-}_{\mathbb{C}}\colon \Omega^{0}(S^{-}_{\mathbb{C}}\otimes (\frg_{E}\otimes\mathbb{C})) \to \Omega^{0}(S^{+}_{\mathbb{C}}\otimes (\frg_{E}\otimes\mathbb{C}))$ denote the $\mathbb{C}$-linear extension of $D^{-}$. We have
\begin{equation*}
\index(D^{-}) = \index_{\mathbb{C}}(D^{-}_{\mathbb{C}}) = - \int_{M} \hat{A}(M) \mathrm{ch}(\frg_{E}\otimes\mathbb{C})\, ,
\end{equation*}
where we have used the index theorem of the chiral complex spin bundle of a $8k$-dimensional manifold coupled to the complexification $(\mathfrak{g}_{E}\otimes \mathbb{C})$ of $\mathfrak{g}_{E}$. Expanding now $\hat{A}(M)$ and $\mathrm{ch}(\frg_{E}\otimes\mathbb{C})$ in terms of the Pontrjagin numbers of $M$ and $\frg_{E}$ we obtain
\begin{equation*}
\index L_A= - (\dim \frg) \hat{A}_2(M) + \frac{1}{24} \la p_1(M), p_1(\frg_E)\ra -\frac{1}{12} (p_1(\frg_E)^2-2p_2(\frg_E))\, .
\end{equation*}
\end{proof}

For a non-integrable $\Spin(7)$-manifold $(M,\Omega)$, we define
\begin{equation} \label{eqn:b27}
b^2_7 = \frac{7p_1(M)^2-4p_2(M)}{5670}- 1 + b^1\, .
\end{equation}
so $\hat{A}_{2}(M)=1-b_1+b_7^2$ and hence the index of $L_{A}$ is formally equal to the index in the $\Spin(7)$-holonomy case, although only in the latter case $b^{1}$ and $b^{2}_{7}$ have the interpretation of the Betti numbers of $M$.

\begin{thm}
\label{thm:frM}
Let ${{A}}\in \frM_{s}$. Then there exist a neighborhood $U([{{A}}])$ of $[{{A}}]$ and a $\Gamma_{A}$-equivariant map
\begin{equation*}
f\colon \HH^{1}_{A} \to \HH^{2}_{A}\, ,
\end{equation*}
such that $U([{{A}}])$ is homeomorphic to $f^{-1}(0)/\Gamma_{A}$.
\end{thm}

\begin{proof}
We have defined the map
\begin{equation*}
\Psi_{{{A}}, \epsilon}\colon T_{{{A}}, \epsilon}\to \Omega^{2}_{7,s-1}(\frg_{E})\, ,
\end{equation*}
whose zero level set $Z(\Psi_{A})$ modulo $\Gamma_{A}$ gives a local model of $\frM_{s+1}$ around $[{{A}}]\in \frM_{s+1}$. Its differential $D \Psi_{{{A}}, \epsilon}\colon 
\ker d^{\ast}_{A} \to \Omega^{2}_{7,s-1}(\frg_{E})$ at every point in $T_{A,\epsilon}$, in particular at $0\in T_{A,\epsilon}$, is Fredholm and has closed range. 
Therefore, there are decompositions
\begin{equation*}
  \ker d^{\ast}_{A} = \ker (\pi_{7}\circ d_{A})^{\perp} \x \HH^1_A\, , \qquad \Omega^{2}_{7,s-1}(\frg_{E}) = \ker (\pi_7\circ d_{A})^{\ast} \x \HH_A^2\, ,
\end{equation*} 
in terms of the hypercohomology groups defined in \eqref{eq:hyperc}. By the Open Mapping theorem $\Psi_{A,\epsilon}|_{V_{0}}\colon 
V_{0}= \ker (\pi_{7}\circ d_{A})^{\perp} \to W_{0}=\ker (\pi_7\circ d)^{\ast}_{A}$ is a Hilbert-space isomorphism. 
Using now Fredholm theory and the fact that $\Psi_{A,\epsilon}(0) = 0$ we conclude that there exist neighborhoods $U(0)\subset T_{A,\epsilon} $ and $V(0) \subset \Omega^{2}_{7,s-1}(\frg_{E})$ of zero such that
\begin{itemize}
\item $U(0) = U_{1}\times U_{2}$ and $V(0) = V_{1}\times V_{2}$ with $U_{2}\subset \HH^{1}_{A}$ and $V_{2}\subset \HH^{2}_{A}$ are (necessarily finite-dimensional) neighborhoods of zero.
	
\item For every $u:=(u_{1},u_{2})\in U_{1}\times U_{2}$ we have	
\begin{equation}
\label{eq:splitPsi}
\Psi_{A,\epsilon}(u_{1},u_{2}) = (F_{1}(u_{1}), F_{2}(u_{1},u_{2}))\, ,
\end{equation}
where $F_{1}\colon U_{1}\to V_{1}$ is a diffeomorphism of Hilbert manifolds and $F_{2}\colon U_{1}\times U_{2}\to V_{2}$ is a differentiable map of Hilbert manifolds.
\end{itemize}
Using equation \eqref{eq:splitPsi} we can characterize $\Psi^{-1}_{A,\epsilon}(0)$ as the pre-image of zero by the differentiable map
\begin{eqnarray*}
 f \colon U_{2}\subset \HH^{1}_{A} &\to& V_{2}\subset \HH_A^{2}\, ,\\
u_{2} &\mapsto & F_{2}(0,u_{2})\, ,
\end{eqnarray*}
and hence we conclude.
\end{proof}

\begin{remark}
For an irreducible connection $[A] \in {\frM}^{\ast}_{s}$ we conclude that there exists a neighborhood $U([A])$ of $[A]$ homeomorphic to 
${f}^{-1}(0)$. If in addition $\HH^{2}_{A} =0$ we deduce that ${\frM}^{\ast}_{s}$ is locally modelled on the vector  $\HH^{1}_{A}$ 
and hence we conclude that ${\frM}^{\ast}_{s}$ is a smooth manifold. In section \ref{sec:transversality} we are going to study how 
generic is the situation for which $\HH^{2}_{A} = 0$  in terms of a generic choice of $\Spin(7)$-structure.
\end{remark}

\subsection*{The case $\frg_{E} = \fru_{E}$  with fixed determinant} \label{sec:moduli-localucase}

In this section we consider the particular case $\frg_{E} = \fru_{E}$, fixing in addition a connection on the determinant line bundle. Let $E\to M$ be a hermitian 
complex vector bundle of rank $r$, over an $8$-manifold $M$ endowed with a $\Spin(7)$-structure given by a four-form $\Omega$. 
The structure group of the vector bundle is $\G=\U(r)$.
As explained in section \ref{sec:Spin(7)rep}, we have a decomposition $\Lambda^2=\Lambda^2_7\oplus \Lambda^2_{21}$ with projections $\pi_7:\Lambda^2\to \Lambda^2_7$ 
and $\pi_{21}\colon\Lambda^2\to \Lambda^2_{21}$.  Let $L=\det E$ be the determinant line bundle, where we fix a connection $\Lambda$. 
Note that there is a decomposition $\fru_E=\RR \oplus \frs\fru_E$, where the $\RR$-summand correspond to the trace of the connection. The space of connections with fixed
determinant is  
  $$
 \cA^\Lambda=\{ A\in \cA \, | \, \tr A=\Lambda\}.
 $$
Fixing $A_0\in \cA^\Lambda$, any other connection $A=A_0+\tau$ has $\tau \in \Omega^1(\frs\fru_E)$. Therefore
$\cA^\Lambda=A_0+ \Omega^1(\frs\fru_E)$. Let $\alpha\in \Omega^1$ be the curvature of $\Lambda$. Any 
connection $A\in \cA^\Lambda$ has curvature splitting as $F_A=(\tr F_A)\Id + F_A^0$, where $F_A^0$ is the 
trace-free part of the curvature. We have that $\tr F_A=\alpha$. So the curvature is defined by 
the variable part $F_A^0\in \Omega^2(\frs\fru_E)$. We give the following definition

\begin{definition}\cite{Tian} \label{def:Spinstanton2}
A conneciton $A\in \cA^{\Lambda}$ is a $\Spin(7)$-instanton if $\pi_{7}(F_A^0) = 0$. 
\end{definition}

Finally, we consider the gauge group 
 $$
 \cG^\Lambda=\{ g\in \cG \, | \, \det g=\Id\},
 $$
which acts on $\cA^\Lambda$, since the action on the determinant line bundle $L=\det E$ is trivial, and hence it does not move the connection.
There is a space of connections modulo gauge 
 $$
 \cB^\Lambda=\cA^\Lambda/\cG^\Lambda\, ,
 $$
and a moduli space of $\Spin(7)$-instantons $\frM^\Lambda =\{[A] \, | \, \pi_7(F_A^0)=0\}$. 

As in the general case we Sobolev-complete $\cA^{\Lambda}$ as to obtain $\cA^{\Lambda}_{s}$, as well as the other spaces.
Given $A\in \cA^{\Lambda}_{s}$ a slice around $A$ for the action of $\cG_{s+1}^\Lambda$ is given by
\begin{equation*}
T_{{{A}},\epsilon} := \left\{ A + \tau\, | \, \tau \in \Omega^{1}_{s} (\frs\fru_{E}) , \, d^{\ast}_{A} \tau = 0 , \, 
\norm{\tau}_{s} < \epsilon\right\}\subset \cA^{\Lambda}_{s}\, .
\end{equation*}
Let $A\in \cA_{s}^\Lambda$ be a $\Spin(7)$-instanton and let $A'={{A}}+\tau \in{\cA}_{s}^\Lambda$ be another connection, 
$\tau \in \Omega^{1}_{s} (\frs\fru_{E})$. Then
 \begin{equation*}
 \pi_7(F_{A'}^0) = \pi_7 \left( d_A\tau+ \frac{1}{2}[\tau, \tau]\right) \, .
 \end{equation*}
Therefore $A+\tau$ satisfies the $\Spin(7)$-instanton equation if and only if 
$\pi_7 \left( d_A\tau+ \frac{1}{2}[\tau, \tau]\right)=0$.
So the theory proceeds now exactly as in the general case, with $\frs\fru_{E}$ playing the role of $\frg_{E}$.
There is a deformation complex 
\begin{equation}\label{eqn:compllex-su}
0\to \Omega^{0}_{s+1}(\frs\fru_{E}) \xrightarrow{d_{A}} \Omega^{1}_{s}(\frs\fru_{E}) \xrightarrow{\pi_{7}\circ d_{A}} \Omega^{2}_{7,s-1}(\frs\fru_E)\to 0\, ,
\end{equation}
with the operator
 \begin{eqnarray*} 
L_{A}\colon \Omega^{1}_s(\frs\fru_{E}) &\to & \Omega^{0}_{s-1}(\frs\fru_{E}) \oplus \Omega^{2}_{7,s-1}(\frs\fru_{E})\, , \nonumber\\
\tau &\mapsto & d_{A}^{\ast}\tau\oplus \pi_{7}\left(d_{A}\tau\right)\, .
\end{eqnarray*}
Proposition \ref{prop:index} about the index of $L_{A}$ can be refined for the special case considered in this section.

\begin{prop} \label{prop:index2}
The index of $L_{A}$ is given by 
\begin{align*}
\index L_A = &  - (r^2-1)(1-b_1+b_2^7) 
+ \frac{1}{24}  \la p_1(M),-2r\,  c_2(E) + (r-1) c_1(E)^2\ra \\
& - \left( \frac{r+7}{12} c_1(E)^4-\frac{r+6}3 c_1(E)^2 c_2(E) + \frac{r+3}3 c_1(E)c_3(E) +\frac{r+6}6 c_2(E)^2 -\frac{r}3 c_4 (E)\right).
\end{align*}
\end{prop}

\begin{proof}
In the particular case of $\frg=\su(r)$, a Chern class calculation gives:
\begin{align*}
&  p_1(\frg_E) = -2r\,  c_2(E) + (r-1) c_1(E)^2\, , \\
& \frac{1}{12}(p_1(\frg_E)^2-2p_2(\frg_E)) = \mathrm{ch}_4(\frg_E)=\mathrm{ch}_4(\End E) =\\
& \qquad = \frac{r+7}{12} c_1(E)^4-\frac{r+6}3 c_1(E)^2 c_2(E) + \frac{r+3}3 c_1(E)c_3(E) +\frac{r+6}6 c_2(E)^2 -\frac{r}3 c_4 (E)\, .
\end{align*}
We substitute in the formula in proposition \ref{prop:index}.
\end{proof}

\begin{remark}
The formula in proposition \ref{prop:index2} specializes to that of reference \cite{Walpuski} when $c_1=0$. 
\end{remark}

\begin{remark}\label{rem:index3}
Proposition \ref{prop:index2} specializes for the case of rank $r=2$ bundles to 
$$
 - 3 (1-b_1+b_2^7) 
+ \frac{1}{24}  \la p_1(M),-4 c_2(E) + c_1(E)^2)\ra  - \left( \frac34 c_1(E)^4-\frac83 c_1(E)^2 c_2(E) +\frac43 c_2(E)^2 \right).
 $$
\end{remark}

\section{Transversality of the moduli space of $\Spin(7)$-instantons}
\label{sec:transversality}

Let $(M,\Omega)$ be a $\Spin(7)$-manifold, and let $E$ be an rank $r$ complex vector bundle endowed with an hermitian metric. Fix a connection $\Lambda$ on the determinant line bundle $L=\det E$ of $E$, and let $A\in{\cA}^\Lambda$ be a $\Spin(7)$-instanton. Associated to it, we have a deformation complex (\ref{eqn:compllex-su}).
We denote by $\HH^0_A,\HH^1_A,\HH^2_A$ its hypercohomology groups. 
Recall that if $A$ is irreducible then $\HH^0_A=0$. We say that $A$ is regular if $\HH^2_A=0$. By theorem \ref{thm:frM}, 
if $A$ is irreducible and regular, then ${\frM}_s^*$ is a smooth manifold around $A$, of finite dimension. 
In general, regularity does not hold. In this section, we shall study in detail how to perturb the equations to get regularity. 

Let $\cS(M):= \Omega^{0}(\S(M))$ be the space of $\Spin(7)$-structures on $M$, namely the space of smooth sections of $\S(M)\subset \Lambda^{4}(M)$. We shall consider tensors of type $C^k$, for some large $k$, and give $\cS(M)$ the $C^k$-topology, so that it becomes a Banach manifold. Associated to each $\Omega\in \cS(M)$ there is a projector $P_\Omega$:
\begin{eqnarray*}
P_\Omega \colon \Omega^{2}_{s-1}(\frs\fru_{E}) &\to & \Omega_{P_\Omega,s-1}^2(\frs\fru_{E})\\
\beta &\mapsto & \frac12 \big( \beta +*_\Omega (\Omega \wedge \beta) \big)\, , 
\end{eqnarray*}
where $\Omega_{P_\Omega,s-1}^2(\frs\fru_{E})$ denotes the space of forms taking values in $\frs\fru_{E}$ of type $\Omega_{7,s-1}^2(\frs\fru_{E})$ with respect to $\Omega$. Let $\Omega_0\in\cS(M)$ be a fixed $\Spin(7)$-structure, and $P_0=P_{\Omega_0}$ the associated projector. For $\Omega$ near $\Omega_0$, the projection
$$
P_0 : \Omega_{P_\Omega,s-1}^2(\frs\fru_{E}) \to \Omega_{P_0,s-1}^2(\frs\fru_{E})
$$
is an isomorphism. Therefore, the equations 
  $$
  P_\Omega(F_A^0)=0 \iff P_0(P_\Omega(F_A^0))=0
  $$
are equivalent, but the second one has a fixed target space. We now consider the map 
\begin{eqnarray*}
\cL: \cA_s \x \cS(M) & \too & \Omega_{P_0,s-1}^2(\su_E) \\
(A,\Omega) & \mapsto & P_0(P_\Omega(F_A^0))
\end{eqnarray*}
which corresponds to a parametric version of the $\Spin(7)$-instanton equation.

To prove that the moduli space for some $\Omega\in \cS(M)$ is regular at any connection $A$ with $\cL(A,\Omega)=0$, we need to prove that $D_1\cL$ is surjective, where $D_1$ denotes the differential with respect to the first variable. The general set up is as follows: suppose $\cU,\cS,\cW$ are Banach manifolds and $F:\cU\x \cS \to \cW$ is a smooth map such that $F_s: \cU\to \cW$, $F_s(x)=F(x,s)$ is Fredholm for any $s\in S$. Suppose that $F$ is a submersion over a point $0\in \cW$, that is $D_{(x,s)} F:T_x \cU \x T_s \cS\to T_0\cW$ is onto for $F(x,s)=0$. Then $\cM=F^{-1}(0) \subset \cU\x \cS$ is a smooth Banach manifold. The projection $\Pi: \cU\x \cS\to \cS$ restricts to a
smooth map $\Pi|_{\cM}: \cM\to \cS$, and 
 \begin{align*}
 & \ker \left( D_{(x,s)} \Pi|_{\cM} \right) \cong \ker D_x F_s \, ,\\
 & \coker \left( D_{(x,s)} \Pi|_{\cM} \right) \cong \coker D_x F_s \, . 
 \end{align*}
So $\Pi|_{\cM}$ is a smooth Fredholm map between Banach manifolds. Recall now the Sard-Smale 
theorem. 

\begin{thm}[{\cite{Smale}}] \label{thm:Smale}
Let $\phi\colon \cM\to \cS$ be a $C^q$-Fredholm map between separable Banach manifolds, with $q>0, q>\index \phi$. Then the set of regular values of $\phi$ is residual in $\cS$. In particular, it is dense.
\end{thm}

Now consider a regular value $s_0\in \cS$ for $\Pi|_{\cM}$. Then $\coker D_x F_{s_0}=\coker \left(D_{(x,s_0)} \Pi|_{\cM}\right)= 0$, for
any $(x,s_0)$ such that $F_{s_0}(x)=F(x,s_0)=0$. This means that $F_{s_0}:\cU\to \cW$ is regular over the point $0\in \cW$,
and $F_{s_0}^{-1}(0)\subset \cU$ is a smooth manifold. 

Thus, going back to our original problem, if we prove that  $D_1\cL$ is surjective whenever $\cL(A,\Omega)=0$, then we have the required transversality 
for a dense set of $\Omega'\in \cS(M)$.

\begin{prop} \label{prop:regularity}
Let $\Omega_0\in \cS(M)$, and let $A\in  \cA_s$ be a $\Spin(7)$-instanton with respect to $\Omega_0$. Let $\psi \in \Omega^2_{7,s-1} (\frs\fru_E)$ be $L^2$-orthogonal to the  image of $D_{(A,\Omega_0)}\cL$. 
Then $d_A^*\psi=0$ and $\tr (F_A^0\wedge \psi)=0$.
\end{prop}

\begin{proof}
Let $\cP:\cS(M) \to \Hom(\Omega^2_{s-1}(M), \Omega^2_{s-1}(M))$ be the map $\cP(\Omega)=P_\Omega$. Then the map $\cL(\Omega,A)=P_0(P(F_A^0))$ has linearization
 \begin{equation}\label{eqn:DL}
   D_{(A,\Omega_0)} \cL(a,\omega)= P_0(d_A a) +  P_0(D_{\Omega_0} \cP(\omega)( F_A^0)).
 \end{equation}
Take $\psi\in \Omega^2_{P_{0}, s-1}(\frs\fru_{E})$ orthogonal to $D_{(A,\Omega_0)} \cL(a,\omega)$ as given in equation (\ref{eqn:DL}). Then 
 \begin{align*}
 & \la P_0(d_A a), \psi \ra =0, \qquad \text{for all }a \in \Omega^1_s(\su_E), \\
 & \la  P_0(D_{\Omega_0} \cP(\omega)( F_A^0)), \psi\ra =0, \qquad \text{for all } \omega\in T_{\Omega_0} \cS(M).
 \end{align*}
The first equation is rewritten $\la d_A a,\psi\ra=0$, since $\psi$ is already in $\Omega^2_{P_0,s-1}(\frs\fru_{E})$. Equivalently $\la a, d_A^*\psi\ra=0$, for all $a \in \Omega^1_{s}(\su_E)$. This means that
 $$
 d_A^*\psi=0.
 $$
The second equation is rewritten as  $\la  D_{\Omega_0} \cP(\omega)( F_A^0), \psi\ra =0$. To simplify it, consider the formula
 $$ 
 \int_M g_\Omega( *_\Omega (\Omega \wedge F_A^0), \psi )=  \int_M \Omega \wedge \tr (F_A^0 \wedge \psi),
 $$
which holds for all $\Omega$, where $g_\Omega$ is the scalar product induced by $\Omega$. We take its derivative at $\Omega_0$ in the direction of $\omega$, and recall that $P_\Omega(\beta)=\frac12(\beta + *_\Omega (\Omega \wedge \beta))$, so 
 $$
\la  D_{\Omega_0} \cP(\omega)( F_A^0), \psi\ra + \int_M D_{\Omega_0} \cG (\omega) (P_0(F_A^0)), \psi ) = \int_M \omega \wedge \tr (F_A^0 \wedge \psi),
 $$
where $\cG:\cS(M)\to \cM\text{et}(M)$ is the map $\cG(\Omega)=g_\Omega$.
As $A$ is a $\Spin(7)$-instanton, we have that $P_0(F_A^0)=0$. Thus the second equation is rewritten as $\int_M \omega \wedge \tr (F_A^0 \wedge \psi)=0$, for all $\omega\in  T_{\Omega_0}\cS(M)$. By proposition \ref{prop:tangent}, $T_\Omega\cS(M) = \Omega^4_1(M) \oplus \Omega^4_7(M) \oplus \Omega^4_{35}(M)$. Therefore $\tr(F_A\wedge\psi)\in \Lambda^4_{27}(M)$. On the other hand, $F_A\in \Lambda^2_{21}(\su_E)$ and $\psi \in \Lambda^2_7(\su_E)$. By proposition \ref{prop:wedges}, $\tr(F_A\wedge \psi) \in  (\Lambda^4_7(M) \oplus \Lambda^4_{35}(M))$. This means that $\tr(F_A\wedge\psi)=0$.
\end{proof}

Now we shall take more general type of perturbations. Fix a background $\Spin(7)$-structure $\Omega_0$,
and therefore also a corresponding metric. Let $\cP(M)$ be the set of all orthogonal projectors $P\colon\Lambda^2(M) \to \Lambda^2(M)$ of rank-seven. As before, we consider tensors of class $C^{k}$, for suitable large $k$, so that $\cP(M)$ becomes a Banach manifold. We consider the projector $P_0\in \cP(M)$ associated to $\Omega_0$, and the perturbed equation
 $$
 P(F_A^0)=0,
 $$
for $A\in\cA^{\Lambda}_s$. Let  
 $$
 \frM_s^P=\{ A\in \cA^{\Lambda}_s \, | \, P(F_A^0)=0\}/\cG_{s+1}
 $$
be the perturbed moduli space. We shall consider $\frM_s^P$ only for $P$ near $P_0$. Note that in this case
 $$
 P(F_A^0)=0 \iff P_0(P(F_A^0))=0,
 $$
and the second equation takes values in a fixed space $\Omega^2_{P_0,s-1}(\frs\fru_{E})$. Consider the functional 
 \begin{eqnarray*}
 \cF_P: \cA_s^\Lambda & \to & \Omega_{P,s-1}^2(\su_E)\, , \\
  A &\mapsto & P(F_A^0)
 \end{eqnarray*}
and the functional $\cF^0_P(A)=P_0(P(F_A^0))\colon \cA_s^\Lambda \to \Omega_{P_{0},s-1}^2(\su_E)$. We want to prove that the moduli space $\frM_s^P$ is regular  at an irreducible connection $A$. This means the surjectivity of $D_A \cF_P$ at any $A$ with $\cF_P(A)=0$, which in turn is equivalent to the surjectivity of $D_A\cF^0_P$.

As before, we consider the parametric version of the perturbed equation, given by the map
  \begin{eqnarray*}
  \cL: \cA_s \x \cP(M) & \too & \Omega_{P_0,s-1}^2(\su_E) \\
   (A,P) & \mapsto & P_0(P (F_A^0)).
 \end{eqnarray*}
To prove that the moduli space $\frM_s^P$, for some $P\in \cP(M)$, is regular at any irreducible connection $A$ with $\cL(A,P)=0$, we need to prove that $D_1\cL$ is surjective, where $D_1$ denotes the differential with respect to the first variable. If that case holds then, by our previous argument, for a dense
second category subset of $P$ near $P_0$ (in the topology of  $\cP(M)$), we will have that $D_A \cF^0_P$ is surjective, for generic $P$, since $\cL(A,P)=\cF^0_P(A)$. This will complete the required transversality.

\begin{prop} \label{prop:transvers}.
Let $\Omega_0\in \cS(M)$, $P_0=P_{\Omega_0}$, and let $A\in  \cA^{\Lambda}_{s}$ be a $\Spin(7)$-instanton with respect to $\Omega_0$. Let $\psi \in \Omega^2_{7,s-1}(\su_E)$ be $L^2$-orthogonal to the  image of $D_{(A,P_0)}\cL$. Then $d_A^*\psi=0$ and $\tr (F_A^0\otimes \psi)=0$, as a section of $\Lambda^{2}_{21}(M)\otimes \Lambda^2_7(M)$.
\end{prop}

\begin{proof}
We fix one projector $P_0$ and a decomposition $\Lambda^2=\Lambda^2_7\oplus \Lambda^2_{21}$.
Other decompositions correspond to the graph of a map
 $$ 
 \mu: \Lambda^2_7 \too \Lambda^2_{21}\, .
 $$
Here the projector is 
 $$
 P=\left(\begin{array}{cc} p & p\, \mu^* \\ \mu \, p & \mu \, p\, \mu^*\end{array}\right),
 $$
where $p=(1+\mu \mu^*)^{-1}$. We consider small perturbations around $\mu=0$, given by some ${\nu}$.
We have
 $$
 \dot P=\left(\begin{array}{cc} 0 & \nu^* \\ \nu  & \nu+\nu^* \end{array}\right).
 $$

The derivative of the map $\cL(A,P)=P_0(P(F_A^0))$ is given by 
 $$
 D_{(A,P_0)} \cL(a,\nu)=P_0 (d_{A}a) + P_0(\dot P(F_A^0))=P_0 (d_{A}a) + \nu^*(F_A^0),
 $$
where $F_A^0\in\Omega^2_{21,s-1}(\su_E)$. Let now $\psi\in \Omega^2_{7,s-1}(\su_E)$ be an
element orthogonal to the image of $D_{(A,P_0)} \cL$. This implies that 
  \begin{align*}
 & \la P_0(d_{A}a),\psi\ra=0\, , \qquad \forall\, a\in\Omega^1(\su_E), \\
 & \la \nu^* (F_A^0), \psi\ra =0\, , \qquad \forall\, \nu\in \Omega^{0}(\Hom( \Lambda^2_7,  \Lambda^2_{21}))\, .
  \end{align*}
The first equation yields that $d_{A}^*\psi=0$. The second equation is equivalent to
 $$
 \int_M \tr(\nu^*(F_A^0) ,\psi) =\la \nu, \tr (F_A^0\otimes \psi) \ra =0,
 $$
for all $\nu$, considered as a section of $\Lambda^2_{21} \otimes  \Lambda^2_7$.
Therefore
  $$
 \tr(F_A\ox \psi)=0.
 $$ 
\end{proof}

The perturbations $P\in \cP(M)$ allow to obtain transversality for the moduli spaces $\frM_s^P$ in the
specific case that $E\to M$ is a rank $2$ vector bundle.

\begin{thm} \label{thm:main1}
Suppose that $E$ is a rank $2$ vector bundle. Let $A\in  \cA^{\Lambda}_{s}$ be a $\Spin(7)$-instanton. Let $\psi \in \Omega^2_{7}(\su_E)$ such that $d_A^*\psi=0$ and $\tr (F_A^0\otimes \psi)=0$. Then $A$ is reducible or $\psi=0$.
\end{thm}

\begin{proof}
The equation $\tr(F_A^0\ox \psi)=0$, where $F_A^0\in \Omega^2_{21}(\su_E)$ and $\psi\in \Omega^2_7(\su_E)$,
means that the (bundle) maps
 $$
 F_A^0:\Lambda^2_{21} \too \su_E, \qquad \psi:\Lambda^2_7 \too \su_E
 $$
have images which are point-wise orthogonal. As
$\su(2)$ has dimension $3$, this implies that either $F_A^0$ is a map of rank $\leq 1$
or $\psi$ is a map of rank $\leq 1$, at any point of $M$ 

Suppose first that $F_A^0=0$ in a ball. Recall that $d_A F_A^0=0$, by the Bianchi identity. As
$*F_A^0 =-\Omega \wedge F_A^0$, we have that
 $$
 d_A^*F_A^0=-*(W\wedge F_A^0),
 $$
where $W=d\Omega$. Then $F_A^0$ satisfies an elliptic equation, 
and hence $F_A^0=0$ everywhere. This means that $A$ is projectively flat, in particular it is not irreducible according to our definition.

Now suppose that $F_A^0$ has rank $1$ in a ball. We trivialize the bundle over the ball, and let $e_1\in \su (2)$ be a unitary
section
spanning the image of $F_A^0$. We complete to an orthonormal basis $\{e_1,e_2,e_3\}$, and write $F_A^0= \omega \otimes e_1$, where $\omega \in \Omega^2_{21}$. 
Write $d_A e_1= \lambda_1 e_1+\lambda_2 e_2+\lambda_3 e_3$, where $\lambda_j\in \Omega^1$. From $d_A F_A=0$ we get
 $$
 0= d_A F_A= (d \omega + \omega \wedge \lambda_1 ) \otimes e_1
+ \omega \wedge \lambda_2 \otimes e_2+ \omega \wedge \lambda_3 \otimes e_3\, .
 $$
Hence $\omega \wedge \lambda_2=0$ and $\omega \wedge \lambda_3=0$.
If $\omega$ is a $2$-form and $\lambda$ is a non-zero $1$-form with $\omega\wedge \lambda =0$, then
$\omega=\lambda\wedge\Theta$ for a $1$-form $\Theta$. But then it cannot be that $\omega\in \Lambda^2_{21}$. 
Therefore $\lambda_2=\lambda_3=0$. This implies that $d_A e_1= \lambda_1 e_1$. As
$e_1$ is unitary, we have that $d_Ae_1=0$. So $A$ is locally
reducible, because we can split $\su_E= \la e_1\ra \oplus \la e_2,e_3\ra$, and $A$ respect both summands.

Next suppose that $\psi$ has rank $1$ in a ball. Write $\psi=\omega \otimes e_1$, where 
$\omega \in \Lambda^2_7$, and $\{e_1,e_2,e_3\}$ is a local orthonormal basis of $\su_E$. 
Write $d_A e_1= \lambda_1 e_1+\lambda_2 e_2+\lambda_3 e_3$. We use the equation $d_A^*\psi=0$ and
the equality  $*\psi=3\, \Omega\wedge\psi$ to get $d_A(\Omega\wedge \psi)=0$. So
 $$
  0=d_A(\Omega \wedge \omega\otimes e_1)= 
  (d(\Omega\wedge \omega)+\Omega\wedge\omega\wedge\lambda_1) \otimes e_1
 +\Omega\wedge\omega\wedge\lambda_2 \otimes e_2+\Omega\wedge\omega\wedge\lambda_3 \otimes  e_3. 
 $$
In particular, $\Omega\wedge \omega \wedge \lambda_2 =0$ and $\Omega\wedge \omega \wedge \lambda_3 =0$. 
The map $\Omega:\Lambda^3 \to \Lambda^7$ has kernel $\Lambda^3_{48}$, so we have
$(\omega \wedge \lambda_2)_{8}=0$, where this is the component in $\Lambda^3_8$. But for any
element $\omega\in \Lambda^2_7$, the map 
 $$
 \omega:\Lambda^1_8 \to \Lambda^3_8
 $$
is an isomorphism; being a map of $\Spin(7)$-representations, it is equivalent to 
Clifford multiplication $V\otimes H \to S^-$, by our discussion of section \ref{sec:Spin(7)rep}. 
Therefore $\lambda_2=\lambda_3=0$. As argued before, we conclude that $A$ is locally reducible.

Finally if $\psi=0$ on a ball, then using that $d_A^*\psi=0$ and $*\psi=3\, \Omega\wedge\psi$, we get
that 
 $$
 d_A\psi =3*( W\wedge \psi).
 $$
So $\psi$ satisfies an elliptic equation and $\psi=0$ everywhere, which is
one of the possibilities in the statement of the theorem. 

Assume that $\psi\not= 0$.
Let us see that the set $U$ of points $x\in M$ such that $A$ is reducible on a ball $B$ around $x$ is dense. 
Let $x_0\in M$. If either $\rk \psi(x_0)=1$ or $\rk F_A^0(x_0)=1$, then $A$ is reducible on a ball around $x_0$ as
argued above. As either $\rk \psi(x_0)\leq 1$ or $\rk F_A^0 (x_0)\leq 1$, then we may assume that one
of them vanishes on $x_0$, say $\psi(x_0)=0$. It cannot be $\psi=0$ on a ball $B$ around $x_0$. If there
is a point $x\in B$ with $\rk \psi(x)=1$, we have $x\in U$, as required. If not, then the set of points with
$\{x\in B| \rk \psi(x)=2\}$ is open and dense in a suitable small ball around $x$. At those points $\rk F_A^0(x)\leq 1$,
but it cannot be that $F_A^0(x)=0$ on all of them, because it is an open set. So there must be some point
with $\rk F_A^0(x)=1$, and this proves that $x\in U$.
Once we have that $U$ is dense, the result follows from the proposition \ref{prop:locally-reducible} below.
\end{proof}

\begin{prop} 
\label{prop:locally-reducible}
Let $A$ be a $\Spin(7)$-instanton which is locally reducible. Then $A$ is reducible (according to definition \ref{def:reducible}). 
\end{prop}

\begin{proof}
Let $A$ be a connection and suppose that it is reducible on a ball. Then the connection is of the form 
 $$
 A=\left(\begin{array}{cc}a i & 0 \\ 0 & -a i \end{array}\right), \qquad a\in \Omega^1.
 $$
The connection on $\su_E$ is of the form 
 \begin{equation}\label{eqn:red}
 A=\left(\begin{array}{ccc} 0 & 0 & 0 \\ 0 & 0 & 2a  \\ 0 & -2a & 0 \end{array}\right).
 \end{equation}
If $a=0$ in an open subset, then $F_A^0=0$ everywhere, as argued in the proof of theorem \ref{thm:main1}. 
This would conclude that $A$ is reducible. Otherwise $a\neq 0$ on an dense
subset of the ball, and then $e_1$ is uniquely determined over the ball (up to sign). We assume this henceforth.

At every point $x$ where $A$ is locally reducible, there is a unique $u$ such that $\nabla_A u=0$. 
Let $R>0$ be the injectivity radius of $M$, and take a ball $B$ of radius $R$ around $x$. Using geodesic coordinates
and parallel transport along radial geodesics, we trivialize the bundle $E$ and the connection $A$. We take a basis
of $\su_E$ at $x$ so that $u=e_1$. The connection is given by a
$1$-form on $B$ of the form
 $$
 A=\left(\begin{array}{ccc} 0 &-\beta & -\gamma \\ \beta &0& -\alpha \\ \gamma & \alpha &0 \end{array}\right).
$$
The local reducibility gives that $A\wedge A=0$ on a dense subset and hence everywhere. On a neighbourhood $B'\subset
B$ of $x$ we have that $\nabla_A e_1=0$. This is equivalent to 
$\beta=\gamma=0$, or $A=\alpha \otimes e_1$, that is, equal to (\ref{eqn:red}). 
Also note that $F_A=d\alpha \otimes e_1$. If we prove that 
$\beta=\gamma=0$ on the whole of $B$, then $A$ is reducible over the larger ball $B$. This happens at every point
$x$, and by density, we will have that there is some section $u\in \Gamma(\su_E)$ with $\nabla_A u=0$, proving 
reducibility on the whole manifold $M$.

Suppose that there are points $y\in B$ such that $A=\delta \otimes f$, for some $1$-form $\delta$ and unitary $f$ with
$\nabla_A f=0$. Note that $\nabla_A f= d f$, so $f$ must be constant on our trivialization, that is $f\in \su(2)$. 
The connection $A$ is $C^\infty$ meaning that $\alpha,\delta$ are $C^\infty$, when extending them by zero outside
the locus where $A$ is in the direction of $e_1,f$ respectively. If we consider the closure of the set where $F_A=d\alpha\otimes
e_1$ and the closure of the set where $F_A=d\delta\otimes f$, then in the intersection we have $F_A=0$ and $A=0$.

So let $V$ be the set where $A$ is locally reducible. Suppose that it has different connected components, 
and take the connected component $W$ that contains $B'$. Over $W$, there is a parallel section $u$,
and $A$ can be written as $A=\alpha \otimes u$, at least locally. The curvature has a global expression
$F_A=d\alpha\otimes u$ over there. That is, $F_A=\omega \otimes u$, where $\omega$ is a closed $2$-form. This form 
can be extended by zero to the complement of $W$, and it is $C^\infty$. Then 
  \begin{align*}
  d\omega &=0, \\
 d^* \omega &= \omega \wedge W,
 \end{align*}
where the first equality follows from the Bianchi identity, and the second one since $*\omega=-(\omega\wedge \Omega)$,
because $A$ is a $\Spin(7)$-instanton.

Therefore $\omega$ satisfies an elliptic equation. If there are other components appart from $W$, then $\omega$ vanishes
in some open set in the complement of $W$. So by elliptic regularity it should be $\omega =0$. This is a contradiction
and completes the proof of the proposition.
\end{proof}

Theorem \ref{thm:main1} implies that for generic $P$ near $P_0$, the irreducible locus $\frM^{P,*}$ is a smooth
manifold of finite dimension given by the index in remark \ref{rem:index3}. 
To argue this, first note that we may take $P$ a $C^\infty$-projector,
since these are dense in the given topology. Consider the equation 
 \begin{eqnarray} \label{eqn:FP}
 \cF_P: \cA_s^\Lambda & \to & \Omega_{P,s-1}^2(\su_E) , \nonumber\\
  A &\mapsto & P(F_A^0).
 \end{eqnarray}
For $A\in \cA_s$ satisfying $\cF_P(A)=0$, we take a slice $T_{A,\epsilon}$ given by the gauge fixing condition
$d_A^*\tau=0$, $\norm{\tau}_s<\epsilon$, where $A'=A+\tau$. This gives a functional
 \begin{eqnarray*}
 L_A^P: \Omega^1_s(\su_E) & \to & \Omega^0_{s-1}(\su_E) \x \Omega_{P,s-1}^2(\su_E) , \\
  \tau  &\mapsto & (d_A^*\tau,  d_A^P \tau),
 \end{eqnarray*}
where $d_A^P=P\circ d_A$. This map $L^P_A$ is  Fredholm for $P$ near $P_0$, since the Fredholm condition is
open. The index of $L_A^P$ is the same as that of $L_A$ (given in proposition \ref{prop:index2}) by the same reason.
Our previous arguments work verbatim for the equation (\ref{eqn:FP}). Therefore, if $L_A^P$ is surjective then
$\frM_s^P$ is a smooth manifold of finite dimension around $[A]$.

Finally, proposition \ref{prop:Cinfty} can also be carried out for the case of (\ref{eqn:FP}), giving that the moduli
space  
 $$
 \frM^P=\{ A\in \cA^\Lambda \, | \, \cF_P(A)=0\} /\cG^\Lambda
 $$
is homeomorphic to $\frM^P\cong \frM^P_s$. Therefore $\frM^P$ is a second-countable, Hausdorff, metrizable
topological space which has the structure of a smooth manifold on the irreducible locus.

\begin{thm}
\label{thm:projectorpert}
For a dense family of projector perturbations $P$ the moduli spaces $\frM^P$ are smooth at irreducible connections, of dimension given by remark \ref{rem:index3}. They are second-countable, Hausdorff and metrizable.
\end{thm}

\subsection*{Holonomy perturbations}
Now we want to give a different type of perturbation that allows to deal with higher rank bundles. These are 
called \emph{holonomy perturbations} are well-known in the context of instantons on $4$-dimensional manifolds \cite{Kronheimer}.

Let $(M,\Omega)$ be an $8$-dimensional manifold endowed with a $\Spin(7)$-structure (not necessarily integrable).
Let $\G$ be a semi-simple compact Lie group, with Lie algebra $\frg$, and let $P\to M$ be a principal $\G$-bundle. As before, consider a faithful
representation of $\G$ and the associated complex vector bundle $E\to M$.  Let $\cA$ be
the space of $\G$-connections on $E$. We perturb the $\Spin(7)$-instanton equation
  $$
   \cF: \cA \too \Omega^2_7(\frg_E), \qquad \cF(A)=F_A
 $$
as follows. 
Consider tuples $(x,\gamma, B, h, \omega)$ where $x\in M$, $\gamma$ is a loop based at
$x$, $B$ is a small ball around $x$, $h:B\x [0,1]\to M$ is a smooth map with
$h|_{B\x\{t\}}$ an embedding of $B$ to a ball centered at $\gamma(t)$, 
and $h|_{B\x\{0\}}=h|_{B\x\{1\}}=\id$, and $\omega$ is a $2$-form on $B$, with
compact support and lying in $\Omega^2_7(\frg_E)$. For each tuple as above, we define
a map 
 $$
 V_{h,\omega}:\cA \too \Omega^2_7(\frg_E).
 $$
as follows. For $A\in \cA$, and for each $y\in B$, consider the holonomy around $\gamma_y(t)=h(y,t)$, 
$h_A(y)=\hol_{\gamma_y}(A) \in \Ad P_y\cong \G \subset \End E_y$. This defines a section $h_A$ of $\Ad(P)$ over $B$. Fix an embedding
$\Ad(P) \cong G \subset M_{n\x n}(\CC)$, and then project orthogonally to $\frg\subset M_{n\x n}(\CC)$, obtaining
a section of $\frg_E$ over $B$. Multiplying by $\omega$, we have an element
 $$
  V_{h,\omega}(A)=\omega \cdot h_A \in \Omega^2_7(\frg_E).
 $$

Take a collection of points $(x_n)$ dense in $M$. For each $x_n$, consider a collection of loops $(\gamma_m)$ dense in the space of loops in
$M$ based at $x_n$, in the $C^1$-topology. Using a diagonal procedure, we obtain a collection $(x_n,\gamma_n)$
of such elements. For each $n$ and $s$, we consider some $\omega_n$ as before with $C^s$-norm bounded by some $c_{n,s}>0$, and 
we require
 $$
 ||c||_s:= \sum_{n=1}^\infty c_{n,s} <\infty.
 $$
The perturbation parameter is $\Theta = \{(x_n,\gamma_n, B_n, h_n, \omega_n) | n\in {\mathbb{N}}\}$, and we denote the space of
such perturbation parameters as $\cW$. Note that this is a Fr\'echet space. Define
 $$
 V_{\Theta}(A)=\sum_{n=1}^\infty V_{h_n,\omega_n}(A).
 $$
This gives a well-defined map
$$
 V_{\Theta} : \cA  \too  \Omega^2_7 (\frg_E).
 $$
Now we perturb the $\Spin(7)$-instanon equations, considering
 $$
 P_{\Theta}:\cA \too \Omega^2_7(\frg_E), \qquad 
 P_{\Theta}(A)=F_A^7 + V_\Theta(A),
 $$

We define the moduli space of \emph{perturbed $\Spin(7)$-instantons} as 
 $$
 \frM^\Theta=\{ A\in \cA \, | \, P_\Theta (A)=0\} /\cG\, .
 $$
The study of the topological properties of $\frM^\Theta$ is similar to that of
$\frM$. To prove smoothness of the irreducible locus of $\frM^\Theta$, we need to study the deformation complex
 $$
 \Omega^0(\frg_E) \stackrel{d_A}\too \Omega^1(\frg_E) \stackrel{d_A^\Theta}\too \Omega^2_7(\frg_E),
 $$
where $d_A^\Theta(a)= d_A^7 a+ D_A V_\Theta(a)$ is the linearization of $P_\Theta$. 
To study it, we consider a Sobolev norm $L^2_s$ and the corresponding gauge group $\cG_{s+1}$, and
space of connections $\cA_s$. For a tuple $(x,\gamma,B,h,\omega)$, the map $V_{h,\omega}$
extends to a smooth map of Banach spaces \cite[Prop.\ 3.1]{Kronheimer}, 
 $$
 V_{h,\omega}:\cA_{s} \too L^2_{s} ( \Lambda^2_7 \otimes \frg_E),
 $$
which satisfies that 
 $$
 ||V_{h,\omega}(A)|| \leq ||\omega||_{C^0}.
 $$
and 
 $$
 ||D_A V_{h,\omega}(a) ||_s \leq K || \omega||_{C^s} ||a||_s\, ,
 $$
for some $K>0$.
Now we consider the space $\cW_s$ of perturbation parameters with $\norm{c}_s<\infty$,
for the given value of $s$. This is now a Banach space. The map
$V_{\Theta}$ extends to 
 $$
 V_{\Theta} : \cA _s \too  \Omega^2_{7,s} (\frg_E).
 $$
This produces a moduli space of \emph{$L^2_s$-perturbed $\Spin(7)$-instantons}
 $$
 \frM^\Theta_s=\{ A\in \cA_s \, | \, P_\Theta (A)=0\} /\cG_{s+1} \, ,
 $$
with $P_\Theta(A)=F_A^7+ V_\Theta(A)$, as before. The deformation complex
 $$
 \Omega^0_{s+1}(\frg_E) \stackrel{d_A}\too \Omega^1_s(\frg_E) \stackrel{d_A^\Theta}\too \Omega^2_{7,s-1}(\frg_E)
 $$
is elliptic, for a small perturbation parameter. The associated map
\begin{eqnarray*}
 L_A^\Theta: \Omega^1_s(\frg_E) & \to & \Omega^0_{s-1}(\frg_E) \oplus \Omega^2_{7,s-1}(\frg_E), \\
  \tau &\mapsto & (d_A^*\tau, d_A^7 \tau + D_A V_\Theta(\tau)),
\end{eqnarray*}
is Fredholm with index given by proposition \ref{prop:index}. 

We aim next to prove that for a dense set of parameters $\Theta$, the map $L_A^\Theta$ is surjective. For this,
we consider the parametric version
 $$
 \cP:\cW_s\x \cA_{s} \too  L^2_{s-1}(\Lambda^2_7 \otimes \frg_E).
 $$
We want to apply the Smale-Sard theorem \ref{thm:Smale}. For this we need the differential at $(A,{\Theta})$,
 \begin{align*}
 & D_1\cP (a,\nu)= d_A^7 a + D_A V_{\Theta} (a) , \\
  &D_2\cP (a,\nu)= V_{\nu} (A).
 \end{align*}
As we have argued before, we only need to see that $D_2\cP$ is surjective at a point $(A,0)$ 
with $A$ an irreducible $\Spin(7)$-instanton.
Note that $D_1\cP$ is Fredholm, since it is the sum of an elliptic operator (which is Fredholm) and
a compact operator (since $D_A V_\Theta$ is bounded from $L^2_s$ to $L^2_s$, it is compact
from $L^2_s$ to $L^2_{s-1}$). So the range of $D_1\cP$ is closed and of finite codimension. 

Take $\psi\in \Omega^2_7(\frg_E)$ in the orthogonal complement of $D_1\cP$.
Recall that we are asuuming that $A$ is irreducible. So for any $x_n$, the holonomies of 
$\gamma_m$ based at $x_n$ generate $\Ad P_{x_n}$, since these loops are dense and the connecction
is irreducible. Therefore perturbing only one $\nu_n$, we have that
$\la \nu \otimes h_A,\psi\ra=0$, for all $\nu$. Hence $\psi(x_n)=0$. By density of the $x_n$, we have
that $\psi=0$ everywhere. 

This concludes that for a dense family of $\Theta$, we have that $L_A^{\Theta}$ is surjective.  
Moreover, we may take $\Theta\in\cW$, by density of $\cW$ in $\cW_s$. In this case, we can
prove, following the same argumental line as before, that $\frM^{\Theta} \cong\frM^{\Theta}_s$. 
Therefore we have proven that 

\begin{thm}
\label{thm:holonomypert}
For a dense family of holonomy perturbations $\Theta$
the moduli spaces $\frM^{\Theta}$ are smooth at irreducible connections, of
dimension given by proposition \ref{prop:index}. They are second-countable,
Hausdorff and metrizable.
\end{thm}

\section{$\Spin(7)$-instantons for line bundles}
\label{sec:instantonline}

When requiring the fixed determinant condition for $\Spin(7)$-instantons, the moduli space of $\Spin(7)$-instantons 
on a line bundle is just a point. However, it is important to undertand the space of solutions to the $\Spin(7)$-instanton 
equation $P_{\Omega}(F_A) = 0$ on a line bundle without the fixed determinant condition, since they appear for reducible 
connections on higher rank bundles. For instance, if $A$ is a reducible connection on a rank-two bundle $E$ with fixed determinant $L=\det E$, 
then $E$ splits as $E=L_1 \oplus (L\otimes L_1^*)$, and the connection $A$ induces a $\Spin(7)$-instanton on the line bundle $L_1$.

Let $M$ be a $\Spin(7)$-manifold with a $\Spin(7)$-structure given by a $4$-form $\Omega$. Let $L\to M$ be a $\mathrm{U}(1)$-bundle, 
with Chern class $c_1(L)\in H^2(M,\ZZ)$, whose image in real cohomology we denote by $c^{\RR}_1(L)\in H^2(M,\RR)$. We want to find the moduli space
space of $\Spin(7)$-instantons. Let $A_0$ be a connection on $L$, with curvature $F_0$. Then $c^{\mathbb{R}}_1(L) = [\frac{F_0}{2\pi i}]\in  H^2(M,\RR)$. 
Other connections $A=A_0+ia$ are given by one-forms $a\in \Omega^1(M)$, so the  space of connections is $\cA=A_0+i\, \Omega^1(M)$. 
The curvature is $F=F_0+i\, da$. The $\Spin(7)$-instanton equation is 
 $$
 P_{\Omega}(A)=\frac12 (F+\ast (F\wedge\Omega) )=0\, .
 $$
The gauge group is simply given by $\cG=\cC^\infty(M,S^1)$. This has connected components parametrized by $[M,S^1]=H^1(M,\ZZ)$. 
The connected component $\cG_0$ of the identity is given by the maps $g=\exp(i\,\theta)$,
$\theta\in \cC^\infty(M)$. Hence $T_{\Id}\cG_0 =i\,\Omega^0(M)$. The action of $\cG_{0}$ on $\cA$ is given by
 $$
 g\cdot A= A+g^{-1}dg = A+ i \, d\theta\, .
 $$
Therefore the orbit of the action of $\cG_0$ on $A\in\cA$ is given by $A+ i \,\im d$.

We complete all spaces of sections with a Sobolev norm $L^2_s$. As $\Omega^1_s(M)=\im d  
\oplus \ker d^{\ast}$, we have that a \emph{global} slice of the action is given by $\ker d^{\ast}$. Therefore the moduli space
 $$
\tilde\frM_L=\{A \in \cA \, | \, P_{\Omega}(A)=0\}/\cG_0
 $$
is given by 
\begin{equation*}
\tilde\frM_L = \left\{ A_{0} + a\, , \,\, a\in \Omega^{1}(M)\,\, | \,\, L(a) = 0 \right\}\, ,
\end{equation*}
where we have defined
\begin{eqnarray}
\label{eqn:L-sec7}
L\colon \Omega^{1}_{s}(M) &\to & \Omega^{0}_{s-1}(M)\oplus \Omega^{2}_{P_{\Omega},s-1}(M)\, , \nonumber\\
a &\mapsto & d^{\ast}a \oplus d^Pa \, ,  
\end{eqnarray}
and $A_0$ is a point in $\tilde\frM_L$. Note that 
the action of $\cG/\cG_0$ on any $A = A_{0} + a$ is given by $a\mapsto a+ \ell$, where $\ell \in \Omega^1(M)$ is an element of 
$\ker d^{\ast} \cap \ker d$, i.e.,\ the harmonic representative of the class of $\cG/\cG_0\cong [M,S^1]\cong H^1(M,\ZZ)$.
This implies that a $b^1$-torus
 $$
  A_0 + H^1(M,\RR)/H^1(M,\ZZ) \subset \frM_L=\{A \in \cA \, | \, P(A)=0\}/\cG
 $$
sits in the moduli space of solutions, for any $A_0\in \frM_L$. This follows since if $a_0$ is a solution of $L(a_{0}) =0$, then $a_0+u$ is also a solution for any $u$ harmonic. 

\subsection*{Integrable case}

First let us suppose that the $\Spin(7)$-holonomy is integrable. Then $F$ is closed and 
$d^*F=*d * F=- *d (F\wedge\Omega)=0$, since $\Omega$ is closed. So $F$ is harmonic. This means
that $F=F_0$, the harmonic representative of $c^{\RR}_1(L)$. When $M$ has holonomy $\Spin(7)$, we have
a decomposition 
  $$
  H^2(M)=H^2_7(M) \oplus H^2_{21}(M),
$$
where $H^2_7(M)$ is the space of harmonic forms $\beta$ with $\beta=3\ast (\beta\wedge \Omega)$
and $H^2_{21}(M)$ is the space of harmonic forms $\beta$ with $\beta=-\ast (\beta\wedge \Omega)$.
Let $b^2_7, b^2_{21}$ be the dimensions of such spaces.

The deformation complex of $\frM_L$ is
$$
 0\to \Omega^0(M) \stackrel{d}\too \Omega^1(M) \stackrel{d^P}\too \Omega^2_{P_{\Omega}}(M)\to 0\, ,
$$
with $d^P=P\circ d$. It has $\HH^0=\RR$, $\HH^1=H^1(M)$, 
$\HH^2=H^2_7(M)$. The first statement is clear. The second follows from the fact 
that $\ker d^P = \ker d$. If $\beta\in \Omega^1(M)$ satisfies
that $d^P \beta=0$. Then $d\beta=-\ast (d\beta \wedge \Omega)$. This 
implies that $d\beta$ is closed and co-closed, hence harmonic. Therefore $d\beta=0$. 
The third statement is proved as follows: any harmonic form $\gamma\in H^2_7(M)$ gives an element in $\Omega^2_{P_{\Omega}}(M)$ with $d^{\ast}\gamma=0$. Hence $\la \gamma ,d\beta\ra=
\la \gamma, P(d\beta)\ra = \la \gamma, d^P \beta \ra=0$, for all $\beta\in \Omega^{1}(M)$. Hence
$H^2_7(M)\to \HH^2$ is injective. Now take an element $\gamma\in \HH^2$. This 
means that $\gamma\in \Omega^2_7(M)$. We project into the orthogonal space to $\im d^P$,
which gives an element representing the same class and $\la \gamma, d^P\beta\ra=
\la \gamma, d\beta\ra=0$, for all $\beta$. Thus $d^*\gamma=0$. As
$*\gamma=-\gamma\wedge\Omega$, we have that $\gamma$ is closed, hence
harmonic, so it lives in $H^2_7(M)$.

Therefore if $c^{\RR}_1(L)\in H^2_{21}(M)$, then there is a connection with harmonic curvature in $\Omega^2_{21}(M)$. The space of solutions is given by connections with harmonic curvature, hence 
 $$
 \frM_L=H^1(M,\RR)/H^1(M,\ZZ).
 $$
If $c_1^{\RR}(L)\not\in H^2_{21}(M)$, then there is no harmonic element representing $c_1^{\RR} (L)$ in
$\Omega^2_{21}(M)$, hence $\frM_L=\emptyset$. Note that when $b^2_7>0$, we have $\HH^2\neq 0$, so the
solutions to the $\Spin(7)$-equations are not regular. This fact, together with the fact that $\dim \HH^{0}=1$ is clearly reflected in the explicit value of the virtual dimension of $\frM_{L}$, given by minus the index of $L$,
\begin{equation*}
\index(L) = 1 - b^{1} + b^{2}_{7}\, .
\end{equation*}
The virtual dimension $-\index(L)$ differs thus from the real dimension $b^{1}$ precisely by the dimension of the vector space $\HH^{0}\oplus \HH^{2}$, as expected. 

If $c^{\RR}_{1}(L) = 0$, namely if $L$ is torsion, then the moduli space of $\Spin(7)$-instantons is always non-empty and we have
\begin{equation*}
\frM_{L} = H^1(M,\RR)/H^1(M,\ZZ)=\Hom(\pi_{1}(M),U(1))\, ,
\end{equation*}
and hence every $\Spin(7)$-instanton is in this case a flat connection.

\subsection*{Non-integrable case}

Now suppose that $\Omega$ is a non-closed $\Spin(7)$-structure, that is, $W=d\Omega$ is possibly non-zero.
Let $A_{0}$ be a connection on $L$ with curvature $F_{0}$, which represents $c^{\RR}_{1}(L)$. The moduli space $\frM_L$ is given by
\begin{equation*}
\tilde\frM_L = \left\{ A_{0} + a\, , \,\, a\in \Omega^{1}(M)\,\, | \,\, L(a) = 0 \right\} ,
\end{equation*}
where we have defined
\begin{equation*}
L_{0}\colon \Omega^1(M) \to \Omega^0(M) \oplus \Omega^2_{P_{\Omega}}(M)\, ,
\end{equation*}
as follows
\begin{equation}\label{eqn:bbb}
 L_{0}(a)=(d^*a, P_{\Omega}(a)),  \quad P_{\Omega}(a)= \frac12( da+ *(da \wedge \Omega))+ \beta_0,
\end{equation}
with $\beta_0 =\frac12(F_0+*(F_0\wedge\Omega))$. Consider the deformation complex
\begin{equation}\label{eqn:def-sec7}
 \Omega^0(M) \stackrel{d}\too \Omega^1(M) \stackrel{d^P}\too \Omega^2_{P_{\Omega}}(M)\, ,
\end{equation}
with $d^P=P_{\Omega}\circ d$, and let $\HH^0$, $\HH^1$, $\HH^2$ be its 
hypercohomology groups. In the non-integrable case, we define $b^2_7$ by the formula (\ref{eqn:b27}), so the index of the complex (\ref{eqn:def-sec7}) equals $1 - b_1 + b^2_7$. In particular, as obviously $\dim \HH^1\geq b^1$ always, then $\dim \HH^2 \geq b^2_7$.

\begin{thm} \label{thm:aaa}
Let $\Omega$ be a $\Spin(7)$-form, and let $W=d\Omega$. Let $\lambda_1>0$
be the smallest non-zero eigenvalue of the Laplacian. If $||W||< \lambda_1$, 
then the space of solutions $\frM_L$, if non-empty, is $H^1(M,\RR)/H^1(M,\ZZ)$.
\end{thm}

\begin{proof}
 Suppose that $A_{0}$ is a $\Spin(7)$-instanton with curvature $F_0$. Then any the other solution $A\in\cA$ to the 
$\Spin(7)$-instanton equation can be written as $A = A_{0} + a$, 
where $a\in\Omega^1(M)$ with $d^*a =0$, and $da=-* (\Omega\wedge da)$. We have 
 $$
 ||da||^2=\int da\wedge *da=-\int da \wedge da \wedge \Omega= - \int a\wedge da\wedge W,
 $$
by integration by parts in the last equality. Suppose that $a$ is orthogonal to the harmonic forms,
and let us see that $da = 0$. We have that $a=G(\Delta a)= G(d^{\ast} da)$, where $G$ is the Green's
operator. Then
 $$
 ||da||^2\leq ||a||\, || da||\, ||W||
 $$
and
 $$
 ||a||=|| G(d^{\ast} d a)|| \leq \frac{1}{\lambda_1} ||d^{\ast} da|| \leq \frac{1}{\lambda_1} ||da|| .
$$
Thus
 $$
 ||da||^2\leq   \frac{1}{\lambda_1}  || da|| ^2 \, ||W||.
 $$
So if $||W||< \lambda_1$, then $da=0$ and hence every $\Spin(7)$-instanton is of the form 
$A = A_{0} + a$ with $a$ harmonic. The result follows.
\end{proof}

\begin{thm}
\label{thm:U(1)moduli}
Under the conditions of theorem \ref{thm:aaa}, we have, for generic $\Omega$,
 \begin{itemize}
 \item If $b^2_7>0$, the moduli space is empty. 
 \item If $b^2_7=0$ or $c_1^{\RR}(L)=0$, then the moduli space is $H^1(M,\RR)/H^1(M,\ZZ)$.
\end{itemize}
\end{thm}

\begin{proof}
Suppose first that $b^2_7=0$. Then the moduli space $\frM_L$ is regular, and hence
of dimension $b^1$. Take $F_0$ a $2$-form representing $c_1^{\RR}(L)$.
Then $\beta_0=P_{\Omega}(F_0)\in \Omega^2_{P_{\Omega}}(M)$. By the surjectivity of $d^P$, there is some $a\in \Omega^1(M)$ with $P_{\Omega}(da)=\beta_0$. This gives a solution to (\ref{eqn:bbb}), which moreover is unique up to an element of $\HH^1$. Thus $\frM_L\cong H^1(M,\RR)/H^1(M,\ZZ)$.

Now suppose that $b^2_7>0$ and $c_{1}^{\RR}(L)\neq 0$. First, let us see that we have regularity for a small perturbation of $\Omega$. By proposition \ref{prop:regularity}, that works exactly in the same way for the case of non-fixed determinant, we have that $\tr(F_A \wedge \psi)=F_A\wedge\psi=0$, where $F_A\in \Omega^2_{21}(M)$ and $\psi\in \Omega^2_7(M)$. Then using remark \ref{rem:ccc}, we have that $\psi=0$ since $F_A\neq 0$ (which follows from $c_1^{\RR}(L)\neq 0$). This completes the claim. Now we have that for a generic nearby $\Omega$, we have regularity of the moduli space $\frM_L$. But then the dimension should be $b^1-b^2_7<b^1$. Therefore theorem \ref{thm:aaa} implies that $\frM_L$ is empty.

Finally, suppose that $c_1^{\RR}(L)=0$. Then we can choose $F_0=0$ and equation (\ref{eqn:bbb}) has solutions. 
Theorem \ref{thm:aaa} gives the statement.
\end{proof}


\end{document}